\documentclass[12pt, oneside]{book}
\usepackage[a4paper,margin=0.5in,footskip=0.05in,headheight=0.05in]{geometry}
\usepackage[utf8]{inputenc}
\usepackage{graphicx}
\usepackage{amssymb}
\usepackage{amsthm}
\usepackage{amsmath}
\usepackage{makeidx}
\usepackage{color}
\usepackage{enumerate}
\usepackage{doi}
\usepackage{xcolor}
\usepackage{yourStyleFile}
\usepackage[backend=bibtex, style=alphabetic, useprefix=true,natbib,maxnames=10,backref=true]{biblatex}
\DeclareCiteCommand{\citeyear}{}					 {\bibhyperref{\printdate}}
{\multicitedelim}
{}

\usepackage{sectsty}

\sectionfont{\large}

\newcommand{\comm}[1]{}
\usepackage{fancyhdr}
\newtheorem{teorema}{Theorem}
\newtheorem{conjetura}[teorema]{Conjecture}

\newtheorem{proposicion}[teorema]{Proposition}
\newtheorem{corolario}[teorema]{Corollary}

\theoremstyle{definition}
\newtheorem{cuestion}[teorema]{Question}
\newtheorem{warning}[teorema]{Warning}

\newtheorem{construccion}[teorema]{Construction}
\newtheorem{definicion}[teorema]{Definition}
\newtheorem{ejemplo}[teorema]{Example}

\newtheorem{observacion}[teorema]{Observation}
\newtheorem{remark}[teorema]{Remark}

\newtheorem{notacion}[teorema]{Notation}

\DeclareMathOperator{\latt}{latt}
\DeclareMathOperator{\lann}{lann}
\DeclareMathOperator{\ann}{ann}
\DeclareMathOperator{\autoann}{autoann}
\DeclareMathOperator{\autolann}{autolann}

\DeclareMathOperator{\evlattid}{evlattid}
\DeclareMathOperator{\lattid}{lattid}

\DeclareMathOperator{\ev}{ev}

\DeclareMathOperator{\soc}{soc}
\DeclareMathOperator{\Evido}{Evido}
\DeclareMathOperator{\Evidb}{Evidb}
\DeclareMathOperator{\MaxEvid}{MaxEvid}
\DeclareMathOperator{\MinEvid}{MinEvid}

\DeclareMathOperator{\maxevid}{maxevid}
\DeclareMathOperator{\minevid}{minevid}
\DeclareMathOperator{\OverEvid}{OverEvid}
\DeclareMathOperator{\BelowEvid}{BelowEvid}

\DeclareMathOperator{\supp}{supp}

\DeclareMathOperator{\nid}{nid}
\DeclareMathOperator{\sdni}{sdni}
\DeclareMathOperator{\cesrb}{cesrb}
\DeclareMathOperator{\evid}{evid}

\DeclareMathOperator{\linspan}{linspan}

\DeclareMathOperator{\Card}{Card}
\DeclareMathOperator{\car}{char}
\DeclareMathOperator{\id}{id}
\DeclareMathOperator{\natidem}{natidem}
\DeclareMathOperator{\mnonnatidem}{mnonnatidem}

\DeclareMathOperator{\con}{con}
\DeclareMathOperator{\stN}{st}
\DeclareMathOperator{\evinf}{evinf}
\DeclareMathOperator{\evsup}{evsup}

\DeclareMathOperator{\Br}{Br}
\DeclareMathOperator{\alg}{alg}
\DeclareMathOperator{\evalg}{evalg}

\DeclareMathOperator{\rid}{rid}
\DeclareMathOperator{\lid}{lid}
\DeclareMathOperator{\tid}{tid}

\DeclareMathOperator{\einf}{inf}
\DeclareMathOperator{\CFM}{CFM}

\DeclareMathOperator{\Aut}{Aut}
\DeclareMathOperator{\MinId}{MinId}
\DeclareMathOperator{\ExMinId}{ExMinId}
\DeclareMathOperator{\InMinId}{InMinId}

\DeclareMathOperator{\sumspan}{sumspan}

\DeclareMathOperator{\idem}{idem}
\DeclareMathOperator{\minidem}{minidem}
\DeclareMathOperator{\snzsc}{snzsc}
\DeclareMathOperator{\chara}{char}

\DeclareMathOperator{\proj}{proj}
\DeclareMathOperator{\msdnonni}{msdnonni}
\DeclareMathOperator{\NonNatMinId}{NonNatMinId}
\DeclareMathOperator{\NatMinId}{NatMinId}

\let\oldquote\quote
\let\endoldquote\endquote
\renewenvironment{quote}[2][]
  {\if\relax\detokenize{#1}\relax
     \def\quoteauthor{#2}%
   \else
     \def\quoteauthor{#2~---~#1}%
   \fi
   \oldquote}
  {\par\nobreak\smallskip\hfill(\quoteauthor)%
   \endoldquote\addvspace{\bigskipamount}}

\begin{document}
\pagenumbering{roman}
\begin{titlepage}
\begin{center}


\includegraphics[width=0.35\textwidth]{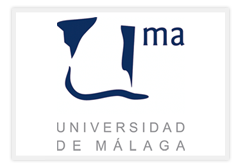}~\\[1cm]

\textsc{\LARGE Universidad de M\'{a}laga}

\textsc{Facultad de ciencias}\\ [1.5cm]

\textsc{\Large Trabajo Fin de Máster}\\[0.5cm]

\noindent\rule{\textwidth}{0.4mm}
{ \huge \bfseries Order Structures around Evolution Algebras: \\[0.4cm]}

\noindent
{ \normalsize \bfseries Inspiring a first idea of the possibility of a schema from the very beginning of a program for a renewed study of induced ordered and binary related sets through different substructures of algebraic structures from a perspective provide by the example of evolution algebras and its diverse socles and ideals }

\noindent\rule{\textwidth}{0.4mm}
\vskip1.5truecm
\textsc{\large\bf Alejandro González Nevado}
\end{center}

\vfill

\begin{center} {\large June 2019}
\end{center}

\end{titlepage}
\newpage
\mbox{}
\thispagestyle{empty}
\begin{titlepage}
\begin{center}

\textsc{\Large Trabajo Fin de Máster}\\[0.5cm]
\bigskip

{ \huge \bfseries Order Structures around Evolution Algebras: \\[0.4cm] }

\noindent
{ \normalsize \bfseries Inspiring a first idea of the possibility of a schema from the very beginning of a program for a renewed study of induced ordered and binary related sets through different substructures of algebraic structures from a perspective provide by the example of evolution algebras and its diverse socles and ideals }

\vskip2truecm

\textsc{\large\bf Alejandro González Nevado}
\end{center}

\vfill

\noindent\emph{Director:} \\
{\bf Prof. PhD. Mercedes Siles Molina}
\bigskip
\bigskip
\bigskip


\vskip.3truecm

\end{titlepage}

\tableofcontents

\section*{Resumen}
Este es un trabajo plenamente de investigación. Consideramos álgebras de evolución y sus subestructuras relacionas: ideales de evolución y subálgebras de evolución. Después de exponer algunos de los conceptos en la literatura relacionados con ellas, exploramos las estructuras de orden que surgen en los conjuntos de subestructuras de un álgebra de evolución. Esto nos lleva a la introducción del zócalo de un álgebra de evolución y al estudio de su conexión con algunos elementos distinguidos dentro del álgebra (tales como los elementos idempotentes y los naturales). Finalmente, examinamos las estructuras de orden que emergen entre estos elementos cuando consideramos las subestructuras que generan dentro del álgebra en que se encuentran. Así, desarrollamos dos enfoques basados en la teoría del orden (una mediante el uso de subestructuras distinguidas y otra mediante el uso de elementos distinguidos). Estas dos estrategias pueden ser utilizadas para distintos objetivos. En particular, nosotros las orientamos hacia el estudio del zócalo de un álgebra de evolución.

\section*{Abstract}
This is a thoroughly research work. We consider evolution algebras and their related substructures: evolution ideals and evolution subalgebras. After exposing some of the concepts related to them in the literature, we explore the order structures that arise in the sets of substructures of an evolution algebra. This leads to the introduction of the socle of an evolution algebra and to the study of its connection with some distinguished elements within the algebra (such as idempotents and natural elements). Finally, we examine the order structures that emerge among these elements when we consider the substructures that they generate inside the algebra. Thus, we develop two order-theoretic approaches (one using distinguished substructures and other using distinguished elements). These two strategies could be used in many ways. Particularly, we direct them towards the study of the socle of an evolution algebra.

\section*{Acknowledgements}
The author of this text is absolutely in debt with the director of this work for her valuable suggestions and corrections during the period of writing of it. Mercedes helped during the writing of this text in many ways (from corrections about style, presentation and results and recommendations of bibliography and papers to emotional and mental support when the things were not working well enough at some point because of some theorems or results being too intricate), but the most important thing is the passion that she feels for mathematics and the ability she has to transmit and show it to her students. I can say that I had the luck and the opportunity of working with her these last months, and the time she dedicated to helping me to finish this text is and was priceless. Thank you.

\section*{Responsibility}
The author of this text is the only and last responsible for all the mistakes, imprecision, inaccuracies or vagueness that could appear along this text. He is also the solely responsible for the opinions expressed in this document. All results, concepts and definitions that do not come from the own work of the author are correctly cited and linked to their original sources without exception. When the ideas are inspired by the ideas or texts of other people by means of private or public conversations, communications, books, papers, texts, conferences, lessons or documents as well as those ideas which find a direct inspiration in material published by others we try to recognize and capture it in the text, even when the connection between these ideas may seem distant although existing. We tried our best to express our acknowledgement and debt with those who preceded us in the marvelous endeavour of the mathematical community through the history of this beautiful intellectual activity.

\newpage

\pagenumbering{arabic}
\setcounter{page}{1}

\section{Introduction, history and current state}

This is the first section of this text. Here we expose briefly the developments which we perform in the next sections. To be clear enough, in every paragraph, we introduce first a schematic explanation of the idea or objective pursued in each one of the three following sections as well as a brief justification of the path taken in this section followed by a detailed exposition of the notions, results and concepts obtained through the route traversed along the definitions and theorems. But before that we give some history and situation.

Since a long time ago it is known that non-associative algebras are an interesting and useful framework for modelling inheritance in genetics. The term \textit{genetics algebra} is the name used to refer to the algebras introduced to study this phenomenon. Evolution algebras are a kind of genetic algebras mainly introduced recently by Tian and Vojtechovsky. In the monograph \cite{tian}, Tian establishes many connections of evolution algebras with different parts of pure and applied mathematics. These relations have led to new research topics about these algebras which are not directly connected with genetics. One of these topics is the proper study of the structures, decompositions and classifications of these algebras in purely mathematical terms which begins mainly in the work \cite{cvs}. There the authors study evolution algebras of arbitrary dimension and they organize the different notions of substructures of these algebras in order to provide valuable examples, definitions, results and connections in different situations. In the second part of that work they study the property of nondegeneracy. The results obtained there will be important for us in the last section of this text.

In the second section of this work we explain and expose the main objects that we will study during the rest of the text. Most notions displayed come from the two sources already cited, although we make slight modifications which we explain for the sake of clarity. These modifications affect mainly to some objects and definitions that are introduced with different names in each one of the cited sources. When this happens, we mainly follow the pattern and terminology fixed in \cite{cvs}. In this sense, it is important to read the discussion after Definition \ref{4}, where we explain the problematic about the names referring to the \textit{evolution} substructures of an evolution algebra, and Observation \ref{10}, where we gather and summarize the most important relations between the notions introduced before closing the second section.

In the third section, we use mainly the objects introduced in the second section as the framework to study order structures arising in evolution algebras. These order structures are specially interesting because they push us to find a new path to follow in the same pattern of the orders appearing. While the set of ideals of an arbitrary algebra present an order structure induced by the inclusion between these ideals, when we deal with the set of evolution ideals of an evolution algebra a different type of order structure appears. Of course, this new type of order structures has resemblances with a lattice but it also has its own features and characteristics which provide us with new nuances to study the order theory hidden behind evolution algebras. Thus, in the third section of this text we deal mainly with the work of studying and motivate these order structures. We try to be as general as possible so we usually do not require the relation with which we are working and is defining the structures to be an order provided that this relation respects the usual notions that we need from an order to work in a lattice-like environment such as some kind of good behaviour towards the extrema (infimum and supremum). Thus, we begin this section introducing the concept of \textit{brset} (Definition \ref{11}) and extending the notion of extrema for orders in the obvious way (Definition \ref{12}). After that, inspired by the situation in lattices of evolution algebras between \textit{evolution} and \textit{non-evolution} substructures, we introduce the highly general Definition \ref{14}, which allows us to introduce a nice family of new binary related structures (which particularize in our case to order structures) suitable for conducting a more general an deeper study than the one concerning just the situation in evolution algebras. The direct particularization of Definition \ref{14} to the case of evolution algebras appears in Definition \ref{15}, where the central concept of \textit{evlattice} is introduced. We point out here that different particularizations of Definition \ref{14} could serve as useful tools in the study of the ordered sets of substructures arising in other mathematical environments such as the relation and structure of prime ideals of rings inside the full set of ideals of a ring. In this sense, the work performed in this section is highly inspired by some fundamental tools in the study of algebraic structures through the exploitation of the information provided by their associated order structures such as \textit{Zariski topologies} or \textit{filters}, although the development of our notions is still in an embryonic stage. Here we also generalize the notion of evolution substructures of an evolution algebra to arbitrary algebras in Definition \ref{17} so we are able to speak about the order structures arising among the evolution substructures in a more general setting. We use Definition \ref{Q} and Definition \ref{21} to introduce the notions of evolution ideals generated \textit{over} and \textit{below} a set in an algebra, respectively. In Theorem \ref{PE} we prove that indeed evolution algebras are a source of examples of the binary related structures introduced before, as we suspected since the beginning of this section. After confirming the abundance of examples of our structures, in Definition \ref{27} we define the notion of extracting lattices from these order structures, which leads to Definition \ref{28} when we come closer to our particular case of study. These notions provoke the necessity of ordering the different structures appearing, so in Definition \ref{29} we introduce an order between lattices respecting the extrema. With this, we are able to establish the unicity result in Theorem \ref{30}, which will be fundamentally exploited in the proof of Theorem \ref{33}, whose reformulation in Theorem \ref{35} closes this section.

Finally, in the fourth and last section, we study the behaviour of idempotent and natural elements and their connection with the socle. We do this in an effort to transport to evolution algebras some results that are central in the study of associative algebras and rings. Unfortunately, we see soon that the fact that evolution algebras are not associative algebras makes the possibility of this extension problematic as a consequence of the widespread use of the associativity in the proofs. We find some counterexamples that show that this transportation cannot be performed as directly as we would like and we need to take care of important aspects about the field underlying the algebra and the systems of quadratic equations appearing when we are searching for idempotents which could generate large parts of the algebras. For these reasons we centre our discussions in simple evolution algebras and focus important parts of our research in the field of scalars underlying them. Finally, we present a result about the idempotency and the socle that we hope could serve as an initial point to find more and deeper decompositions of this type through the study and exploitation of the connections among the notions and concepts propounded along this text. Hence, we begin the section remembering some usual concepts in Definition \ref{37} in order to cite Proposition \ref{41} from \cite{cvs} about the relations between the notions introduced in the Definition \ref{37}. After that, we remember the definition of \textit{minimal idempotent} in Definition \ref{42} in order to be able to cite the central and widely known Theorem \ref{F} about the relation between idempotents and minimal one-sided ideals in semiprime associative rings, which is our main leitmotiv through this section. Our objective and main purpose is to find an analogue for this result within the framework of evolution algebras. However, we find soon an obstacle: the lack of associativity hampers the possibility of transporting the reasoning used in the proof of that result to the environment of evolution algebras. In fact, in Example \ref{Exa}, we find directly a counterexample to its literal translation to the language of evolution algebras and it seems that many more counterexamples of this type can be found in a similar way. After that, we proceed with Construction \ref{B} in order to find general conditions implying the presence of non-zero idempotents in simple evolution algebras. From Construction \ref{B} we deduce the convenience of introducing Definition \ref{DDD} and Definition \ref{56} about the systems of equations appearing in the construction. In this discussion, Theorem \ref{THE} allows us to simplify some conditions appearing in the systems of equations of the construction. This lead us finally to Definitions \ref{58} and \ref{59}, where we introduce the fields that will underlie our algebras and the \textit{good} choices of structure matrices allowing the appearance of idempotents in the evolution algebras that they generate. Conjecture \ref{Con} closes this discussion centred mostly around the relations between the fields underlying the algebras and their structure constants. The socle and the evolution socle are defined in Definition \ref{61}, which is strongly tied to the discussion that encompasses Warning \ref{62} and Question \ref{63} about the particular behaviour of these structures. After some discussion about the zoo of evolution algebras and some examples, we begin with the construction of idempotents with Definition \ref{69} and Propositions \ref{70} and \ref{71}. In Theorem \ref{75} we locate minimal idempotents through the use of the different \textit{annihilators} introduced in Definition \ref{72}. In Definition \ref{76} we introduce evolution algebras having minimal ideals which are simple considered themselves as algebras. These algebras will be central for us as we will see in Theorem \ref{80} and Theorem \ref{108}. The introduction of Definition \ref{81} allows us to make interesting computations about minimal ideals in Proposition \ref{82}. After that, we pass to the consideration of \textit{natural elements} in Definition \ref{84} and study their behaviour in Proposition \ref{85}. We study the orthogonality of idempotents from Definition \ref{86} to Theorem \ref{88} before considering the notion of unicity for natural bases of an evolution algebra in Definition \ref{89} and some results related with this concept and the concept of counting natural bases until Corollary \ref{103}. Finally, we introduce notions related to the \textit{socle of natural idempotency} in Definition \ref{104} and Definition \ref{106} which will lead us to the proof of Theorem \ref{108} about a suitable decomposition of the socle closing this work.

In the end, we do not conceive this as a closed work. On the contrary, here we propose the beginning of a program towards the study of the order structures arising around the sets of substructures of algebraic objects. We hope that all the study and research conducted here can be valuable for future work and developments within this field and not just the end of a piece of text.

\section{Preliminaries about algebras, evolution algebras and directly related structures and substructures}

First of all we remember some definitions that we will use through this text. It is very important to establish these concepts precisely enough.

\begin{definicion}\cite{th}\cite{sl}\cite{ma}\cite{rds}\cite{lli}
Given a field $\mathbb{K},$ an \textbf{algebra} $A$ over $\mathbb{K}$ is a vector space over $\mathbb{K}$ with a bilinear map $A\times A\to A, (a,b)\mapsto ab$ called the \textbf{multiplication} of $A.$ For an algebra, we have the following terminology.

\begin{enumerate}
    \item A \textbf{subalgebra}, $B,$ of an algebra over $\mathbb{K},$ $A,$ is a linear subspace which verifies that the multiplication of any two of its elements is again in the subspace. In other words, a subalgebra of an algebra is a subset of elements that is closed under addition, multiplication and scalar multiplication. In symbols, we say that a subset $B$ of an algebra $A$ over $\mathbb{K}$ is a subalgebra if, for every $a,b\in B$ and $k\in\mathbb{K},$ we have that $ab, a+b, ka\in B.$ We denote it by $B\leq_{\alg}A.$
    \item A \textbf{left ideal}, $L,$ of an algebra, $A,$ over $\mathbb{K}$ is a linear subspace of $A$ which verifies that any element of the subspace $L$ multiplied on the left by any element of the algebra $A$ remains in the subspace $L.$ Symbolically, we say that a subset $L$ of the algebra $A$ over $\mathbb{K}$ is a left ideal if, for every $a,b\in L,$ $c\in A$ and $k\in\mathbb{K},$ we have that $a+b, ka, ca\in L;$ we denote it by $L\vartriangleleft_{l}A.$ If, in the last expression, we substitute $ca$ by $ac,$ then we have the definition of a \textbf{right ideal}; we denote it by $L\vartriangleleft_{r}A.$ A \textbf{two-sided ideal}, $L,$ is a subset that is both a left and a right ideal; we denote it by $L\vartriangleleft_{t}A.$ Notice that every ideal is a linear subspace and a subalgebra of $A.$
    \item Given a subset $S$ of the algebra $A,$ we define the following concepts. The \textbf{subalgebra (left ideal, right ideal, (two-sided) ideal) generated by $S$} is the minimum subalgebra (left ideal, right ideal, (two-sided) ideal), $B,$ of $A$ containing $S$ as a subset; we denote it by $\alg(S)$ ($\lid(S),\rid(S),\tid(S)$). Observe that $\alg(S)=\bigcap_{S\subseteq C, C\leq_{\alg}A}C.$ We denote $\mathcal{L}_{\alg}^{A}(S)$ the \textbf{lattice of subalgebras of $A$ containing $S$} with the operations defined as follows: for every $X, Y\in\mathcal{L}_{\alg}^{A}(S),$ $\sup(X, Y):=\alg(X\cup Y)$ and $\einf(X, Y):=X\cap Y$ are the supremum and the infimum, respectively. We also denote the order induced by the inclusion in the lattice $\mathcal{L}_{\alg}^{A}(S)$ by $\leq_{\alg}^{S}$ or just $\leq_{\alg}$ when the subset $S$ is understood. We have the definitions for the corresponding concepts of \textbf{left ideal, right ideal} and \textbf{two-sided ideal} just repeating exactly the previous one but changing the term ``subalgebra" by ``left ideal", ``right ideal" and ``two-sided ideal" and substituting $\alg$ by $\lid,\rid$ and $\tid,$ respectively.
    \item The algebra $A$ is \textbf{commutative} if $ab=ba$ for every $a,b\in A.$
    \item The algebra $A$ is \textbf{associative} if $(ab)c=a(bc)$ for every $a,b,c\in A.$ If $A$ is associative, then, we can write $(ab)c=abc=a(bc)$ without any risk of confusion.
    \item The algebra $A$ is \textbf{flexible} if $(ab)a=a(ba)$ for every $a,b\in A.$
    \item The algebra $A$ is \textbf{power associative} if the subalgebra generated by any element $a\in A,$ $\alg(\{a\}),$ is associative.
\end{enumerate}
\end{definicion}

During all this text, we will consider our algebras to be finite-dimensional as vector spaces by default. Thus, unless we say otherwise explicitly, all our algebras verify that their dimensions as vector spaces are natural numbers.

Once we have set in the previous definitions the basic concepts that we will use, it is the time to fix the main objects concerning our research. First, we begin with the concept of evolution algebra.

\begin{definicion}\cite{tian}
Let $\Lambda$ be an index set and $\mathbb{K}$ a field. An \textbf{evolution algebra} over a field $\mathbb{K}$ is an algebra $A$ over $\mathbb{K}$ provided with a basis $B=\{e_{i}\mid i\in\Lambda\}$ such that $e_{i}e_{j}=0$ for $i\neq j.$ A basis verifying this condition is called a \textbf{natural basis} of $A.$ Fixed a natural basis $B$ in $A,$ the scalars $\omega_{ki}\in\mathbb{K}$ such that $e_{i}^{2}:=e_{i}e_{i}=\sum_{k\in\Lambda}\omega_{ki}e_{k}$ are the \textbf{structure constants} of $A$ \textbf{relative to} $B.$ This allows us to construct the \textbf{structure matrix of} $A$ \textbf{relative to} $B,$ $M_{B}(A):=(\omega_{ki})_{i,k\in\Lambda}.$ We fix $a\in A$ arbitrary such that $a=\sum_{i\in\Lambda}a_{i}e_{i}$ with $a_{i}\in\mathbb{K}$ for all $i\in\Lambda.$ We define the support of $a$ in the basis $B$ as the set $\{i\in\Lambda\mid a_{i}\neq 0\}$ and denote it by $\supp_{B}(a).$ Clearly, $\Card(\supp_{B}(e_{i}^{2}))=\Card(\{k\in\Lambda\mid\omega_{ki}\neq0\})\in\mathbb{N}$ for every $i;$ we call this number the \textbf{cardinal of the expression of the square of $e_{i}$ relative to the basis $B$} or, simply, \textbf{cardinal of the support of the square of $e_{i}$ in the basis $B$} and we denote it by $\cesrb(e_{i},B):=\Card(\supp_{B}(e_{i}^{2})).$ Therefore $M_{B}(A)\in\CFM_{\Lambda}(\mathbb{K}),$ where $\CFM_{\Lambda}(\mathbb{K})$ is the vector space of those matrices $\Lambda\times\Lambda$ over $\mathbb{K}$ for which every column has at most a finite number of non-zero entries.
\end{definicion}

\begin{notacion}
Let $A$ be a vector space and $T\subseteq A.$ We denote $\linspan(T)$ the minimum vector subspace of $A$ containing $T.$
\end{notacion}

It is easy to see that, by definition, every evolution algebra is commutative; and, for this reason, the concepts of left, right and two-sided ideals are exactly the same within this environment. As a consequence, henceforth, we develop the theory only referring to two-sided ideals, which we will call, from now on, just \textit{ideals} (except that we say otherwise). Nevertheless, we cannot say in general that evolution algebras are associative; not even in the cases of small dimension $n\leq 4$ or over finite fields, as we can see in the examples, developments, descriptions and classifications made in the articles \cite{mddcm}, \cite{aym}, \cite{ymm}, \cite{fn}, \cite{ub}, \cite{fm} and \cite{ymv} (more detailed in \cite{ymv2}). Now that we have our main structure, it is the time to define its substructures. We will do this making a discussion of the work of Tian \cite{tian} following the suggestions appearing in \cite[From Definitions 2.4 to Example 2.11]{cvs}.

\begin{definicion}\label{4}
An \textbf{evolution subalgebra} of an evolution algebra $A$ is a subalgebra $A'\leq_{\alg}A$ such that $A'$ is an evolution algebra, i.e., $A'$ has a natural basis; we denote it as $A'\leq_{\evalg}A.$ We say that a subset $S\subseteq A$ verifies the \textbf{subset linear extension condition} if there exists a subset $T\subseteq S$ verifying $S\subseteq\linspan(T)$ such that $T$ can be extended to a natural basis of $A.$ We say that the evolution subalgebra $A'$ verifies the \textbf{evolution subalgebra extension condition} if there exists a natural basis $B'$ of $A'$ which can be extended to a natural basis of $A.$ We say that a subalgebra $A''$ of $A$ verifies the \textbf{subalgebra extension condition} if there exists a natural basis $B''$ of $A''$ which can be extended to a natural basis of $A.$
\end{definicion}

In the seminal work of Tian on evolution algebras \cite{tian}, the names used in the previous definition refer to essentially different substructures, as pointed out in \cite[Discussion between Remark 2.5 and Example 2.6]{cvs}. Nevertheless, we will introduce the concepts defined by Tian in \cite[Definition 4 in 3.1.3 Basic Definitions]{tian} as \textit{extension evolution subalgebras} and \textit{extension evolution ideals} in the same the spirit of \cite[Definitions 2.4]{cvs}.

The following results are immediate from the definitions.

\begin{proposicion}
Let $A'$ be a subalgebra of the evolution algebra $A.$ Then the following assertions are equivalent.
\begin{enumerate}
    \item The subset $A'$ of $A$ verifies the subset linear extension condition (considering $A'$ as a subset of $A$).
    \item The evolution subalgebra $A'$ of $A$ is an evolution subalgebra verifying the evolution subalgebra extension condition (considering $A'$ as an evolution subalgebra of $A$).
    \item The subalgebra $A'$ of $A$ verifies the subalgebra extension condition (considering $A'$ as a subalgebra of $A$).
\end{enumerate}
 \end{proposicion}

Using this proposition we can define an important kind of evolution subalgebras: extension evolution subalgebras.

\begin{definicion}\cite[Definitions 2.4]{cvs}
A subalgebra $A'$ of the evolution algebra $A$ which verifies the subalgebra extension condition is called an \textbf{extension evolution subalgebra}.
\end{definicion}

Now, as we just did with the concept of subalgebra, we extend the notion of ideal to the environment of evolution algebras suitably.

\begin{definicion}\cite[Definition 2.8]{cvs}
Let $A$ be an evolution algebra. An \textbf{evolution ideal} of $A$ is an ideal $I$ of $A$ such that $I$ has a natural basis; we denote it as $I\vartriangleleft_{\evid}A.$ We say that the evolution ideal $I$ verifies the \textbf{evolution ideal extension condition} if there exists a natural basis $B'$ of $I$ which can be extended to a natural basis of $A.$ We say that an ideal $J$ of $A$ verifies the \textbf{ideal extension condition} if there exists a natural basis $B''$ of $J$ which can be extended to a natural basis of $A.$
\end{definicion}

As in the case of evolution subalgebras, now we have a straightforward proposition relating the three extension conditions seen so far.

\begin{proposicion}
Let $I$ be an ideal of the evolution algebra $A.$ Then the following assertions are equivalent.
\begin{enumerate}
    \item The subset $I$ of $A$ verifies the subset linear extension condition (considering $I$ as a subset of $A$).
    \item The evolution ideal $I$ of $A$ is an evolution ideal verifying the evolution ideal extension condition (considering $I$ as an evolution ideal of $A$).
    \item The ideal $I$ verifies the ideal extension condition (considering $I$ as an ideal of $A$).
\end{enumerate}
 \end{proposicion}

\begin{definicion}\cite[Definitions 2.4 and 2.8 and Remark 2.10]{cvs}\label{EEI}
An ideal $I$ of the evolution algebra $A$ which verifies the extension condition for ideals is called an \textbf{extension evolution ideal}.
\end{definicion}

Let $A$ be an evolution algebra. It is immediate that every extension evolution ideal of $A$ is an evolution ideal verifying the evolution ideal extension condition. However, as it is shown in \cite[Example 2.11]{cvs}, extension evolution ideals do not coincide in general with evolution ideals, i.e., not every evolution ideal verifies the evolution ideal extension condition. Particularly, it is immediate that every extension evolution subalgebra of $A$ is an evolution subalgebra verifying the evolution subalgebra extension condition, but, as every ideal of $A$ is a subalgebra of $A$, by the same cited example \cite[Example 2.11]{cvs}, extension evolution subalgebras do not coincide with evolution subalgebras, i.e., not every evolution subalgebra verifies the evolution subalgebras extension condition. Moreover, as it is shown in \cite[Example 2.6]{cvs}, an evolution subalgebra is not necessarily an ideal, and, therefore, an evolution subalgebra is not necessarily an evolution ideal either. However, as it is shown in \cite[Proposition 2 in 3.1.4 Ideals of an evolution algebra]{tian} although with a different terminology, every extension evolution subalgebra is an extension evolution ideal, and, therefore, both concepts are the same.
Another interesting conclusion to be pointed out is that not every ideal of an evolution algebra has a natural basis, i.e., not every ideal is an evolution ideal \cite[Example 2.7]{cvs}. And, immediately for this reason, not every subalgebra of an evolution algebra has a natural basis, i.e., not every subalgebra is an evolution subalgebra.
Thus, we recall that ideals and evolution ideals as well as subalgebras and evolution subalgebras do not coincide in general. Furthermore, \cite[Example 2.3]{cvs} and \cite[Example 2.7]{cvs} show that the class of evolution algebras is not closed neither under the operation of taking subalgebras nor under the operation of taking ideals, respectively, as it is mentioned at the beginning of \cite[Subsection 2.1. Subalgebras and ideals of an evolution algebra]{cvs} just before the first example \cite[Example 2.3]{cvs}.

We summarize all the discussion of the work done here, introduced in \cite[Subsections 3.1.3. Basic definitions and 3.1.4. Ideals of an evolution algebra]{tian} and developed in \cite[Subsection 2.1. Subalgebras and ideals of an evolution algebra]{cvs} in the following observation.

\begin{observacion}\label{10}
Let $A$ be an evolution algebra. The notions of extension evolution ideal, evolution ideal, ideal, subalgebra and evolution subalgebra are essentially pairwisely different. However, extension evolution subalgebras and extension evolution ideals are the same concept.
\end{observacion}

All these relations between the six substructures exposed lead us to consider some of the constructions and structures that we discuss and study during the next sections.

\section{A first approach to the lattice-like structures arising in the different sets of substructures introduced in the previous section of an evolution algebra and a suggestion to go beyond the evolution algebra framework}

As we have seen in the previous section, in the general setting of algebras and rings, ideals are subsets of these structures that are special in the important sense that we can define easily their sums, intersections and products, which are also ideals. In fact, ideals are a kind of \textit{completions} of subalgebras that allow us to do this. But with evolution ideals we lose again these important features of ideals and we will try to recover them in some sense related to the behaviour of the order structures induced by the inclusion in the sets formed by all the substructures of one particular kind that one given algebra possesses. Nevertheless, we develop the next order-theoretic notions inspired by evolution ideals of evolution algebras because our purpose during the next section will be developing the theory within this frame so we are able to describe and study the socle of an evolution algebra using them.

In the previous section we have defined some substructures related to an evolution algebra and to an arbitrary algebra. Now we want to study how, given a fixed evolution algebra $A,$ these substructures interact.

As we want to be as general as possible, we introduce a concept that generalizes the term \textit{poset} beyond partial ordered sets to a broader class of relations. Also, we extend some of the usual constructions and operations performed with partial orders to general binary relations. For our immediate next developments we are fundamentally inspired and in debt with the research in relations exposed in the books \cite{rel1}, \cite{rel2} and \cite{rel3}.

\begin{definicion}\label{11}
Let $(\Omega,R)$ be a set $\Omega$ together with a binary relation $R.$ We say that $(\Omega,R)$ (or simply $\Omega$ if $R$ is clear from the context) is a \textbf{binary related set} or \textbf{brset}. We denote $\con(R)$ the \textbf{converse} relation of $R$ and $\stN(R)$ the \textbf{strict} relation of $R.$ That is, $(a,b)\in R$ if and only if $(b,a)\in\con(R),$ and $(a,b)\in \stN(R)$ if and only if $(a,b)\in R$ with $a\neq b.$ Usually, as it is common with binary relations, we express these memberships as $aRb,$ $b\con(R)a$ and $a\stN(R)b,$ respectively.
\end{definicion}

For our purpose, we slightly redefine the usual ``supremum" and ``infimum" to extend them from orders to arbitrary binary relations. We mention that, although this is not developed here, it is obvious that the same generalization process can be performed for other distinguished elements within orders important in order theory towards a more general treatment of binary relations. These generalization could be important and lead to some new and unknown aspects of binary relations and orders. However, we do not develop this work here because, in some cases, it becomes more complicated and tedious, and because the two notions presented and generalized in what follows fulfill our objectives for now. The generalization is direct.

\begin{definicion}\label{12}
Let $(\Omega,R)$ be a brset and $S\subseteq\Omega.$ We define the next objects.
\begin{enumerate}
    \item A \textbf{lower bound} of the subset $S$ is an element $a\in\Omega$ such that $aRx$ for all $x\in S.$ A lower bound $a$ of $S$ is called an \textbf{infimum} of $S$ if, for all lower bounds $y\in\Omega$ of $S,$ we have that $yRa.$
    \item An \textbf{upper bound} of the subset $S$ is an element $b\in\Omega$ such that $xRb$ for all $x\in S.$ An upper bound $b$ of $S$ is called a \textbf{supremum} of $S$ if, for all upper bounds $z\in\Omega$ of $S,$ we have that $bRz.$
\end{enumerate}

\end{definicion}

To continue, we need first some notation.

\begin{notacion}
Let $\mathcal{P}(\Omega)$ denote the power set of $\Omega,$ i.e., the set of all subsets of $\Omega.$ Let $\Br(\Omega,R)$ (we usually omit the relation $R$ when it is clear from the context and simply write $\Br(\Omega)$) denote the set of all non-empty subbrsets of $\Omega,$ i.e., all the non-empty subsets $A$ of $\Omega$ where $A$ is automatically considered equipped with the restriction of the relation $R$ to this subset, $R\restriction_{A},$ so we have the brset $(A,R\restriction_{A}).$ We also denote $\binom{E}{n}$ the set of all subsets of $E$ having $n$ elements.
\end{notacion}

For what follows in this paper, we could just have taken $\mathcal{P}(\Omega)$ instead of introducing $\Br(\Omega,R).$ Nevertheless, we consider that the notions and concepts presented here could be extended in some future research to domains where the induced substructures in the set of parts (subsets) is more determinant for the development of the theory related to these objects. Thus, we keep this more suggestive name.

Now we remember the definition of lattice and we expand it to cover our interests. As this definition is classically supported in the definition of supremum and infimum, the previous redefinition of these concepts is enough for our general purposes and we can just reproduce the classical definition having in mind that our infima and suprema are now more general. Also, for clarity, instead of $R$ we will use the usual order-relation symbols remembering that we are considering them in a broader sense. Thus, the beginning part about lattices in our next definition comes from \cite[Definition 6.14, Section 1.6. Neighboring elements. Bounds]{rel4} or \cite[Definition 2.4(i)]{rel5}, \cite[Chapter 2]{rel6}. However, we should consider it having a broader sense than the sense given in these books, which are main references in order theory. In addition, we extend the definition of ``lattice" towards new concepts which help us covering new perspectives in the order-theoretic structures emerging inside the sets of substructures of an evolution algebra, as we will see soon.

We begin with a highly general construction.

\begin{definicion}\label{14}
Let $(L,\leq_{L})$ be a brset and $H$ a set of maps $h\colon A\to\Br(L)$ defined from subsets $A$ of $\Br(L)$ to $\Br(L).$ We say that $(L,\leq_{L},H)$ is a \textbf{mapped brset}. Consider a subbrset $(E,\leq_{E}),$ with $E\subseteq L$ and $\leq_{E}$ the restriction of $\leq_{L}$ to $E,$ $F$ a set of maps $f\colon B\to\Br(E)$ defined from subsets $B$ of $\Br(E)$ to $\Br(L),$ and a logical property $\varphi(L,\leq_{L}, H, E, \leq_{E}, F)$ depending on the previous objects. Then we say that $(E,\leq_{E},F)$ is a \textbf{$\varphi$-brset over the mapped brset $(L,\leq_{L},H)$} or that $(E,\leq_{E})$ is an \textbf{$F,\varphi$-brset over the mapped brset $(L,\leq_{L},H)$} if $\varphi(L,\leq_{L}, H, E, \leq_{E}, F)$ is true for this choice of parameters viewed as logical variables for the logical property $\varphi.$
\end{definicion}

Obviously, we are just assembling one mapped brset over another one which is acting as the basis of the assembly. There are three direct ways of generalizing this concept and go beyond our previous definition. First of all, we could assemble more than just one mapped set to the basis mapped set $(L,\leq_{L},H).$ Also, we could assemble mapped brsets without a basis as $(L,\leq_{L},H)$ and letting all the work of the connection among components of the assembly to the logical property $\varphi.$ Finally, we could assemble brsets in sequence (or, more general, following the order of the inclusion between subsets of a set) so we assemble first $E$ over the basis $L$ and then we use $E$ as a basis to assemble another subset $E'\subseteq E$ and continuing this process of construction by induction. We will not develop these further generalization of the previous concept here because it is beyond the scope of this work and the definition just introduced is enough for our purposes of showing the importance of consider the structure arising within certain sets of substructures of an evolution algebra. Furthermore, we note that the study of these structures arising in the equivalent constructions of the sets mentioned here in rings, groups and other algebraic structures and the development of the construction of assemblies of mapped brsets could be an interesting object of study and research as well as the relation and connection between these two developments; here we will just show an small superficial part of what could be a hidden iceberg in the ocean of algebra.

After some trivial particularizations applied to the objects used and established in the previous definition, our main objects of study lead us to the next more concrete concepts. Thus, the next definition is just a suitable and useful particular case of the previous more general definition.

\begin{definicion}\label{15}
A brset $(L,\leq_{L})$ is called a \textbf{lattice} if each two-element subset $\{a,b\}\subseteq L$ has supremum and infimum. When a brset $(L,\leq_{L})$ just verifies that each two-element subset $\{a,b\}\subseteq L$ has supremum (infimum) we call it a \textbf{supremum-lattice (infimum-lattice)}. Consider a subbrset $(E,\leq_{E}),$ with $E\subseteq L$ and $\leq_{E}$ the restriction of $\leq_{L}$ to $E.$

\begin{enumerate}
    \item Consider a map $f\colon\binom{E}{2}\to\Br(L).$ The triple $(E,\leq_{E},f)$ is called an \textbf{infevlattice} over the infimum-lattice $(L,\leq_{L},\inf)$ if the infimum $i$ in $(L,\leq_{L})$ of any two-element subset $\{a,b\}\subseteq E$ verifies that, for every $x\in f(\{a,b\}),$ $x\leq_{L} i.$ We call $f(\{a,b\})$ the \textbf{front of infima} of the set $\{a,b\}$ in $(E,\leq_{E},f).$
    \item Consider a map $g\colon \binom{E}{2}\to\Br(L).$ The triple $(E,\leq_{E},g)$ is called an \textbf{supevlattice} over the supremum-lattice $(L,\leq_{L},\sup)$ if the supremum $s$ in $(L,\leq_{L})$ of any two-element subset $\{a,b\}\subseteq E$ verifies that, for every $x\in g(\{a,b\}),$ $s\leq_{L} x.$ We call $g(\{a,b\})$ the \textbf{rear of suprema} of the set $\{a,b\}$ in $(E,\leq_{E},g).$
    \item The quadruple $(E,\leq_{E},f, g)$ is called an \textbf{evlattice} over the lattice $(L,\leq_{L},\inf,\sup)$ if $(E,\leq_{E},f)$ is an infevlattice over the infimum-lattice $(L,\leq_{L},\inf)$ and $(E,\leq_{E}, g)$ is a supevlattice over the supremum-lattice $(L,\leq_{L},\sup).$
\end{enumerate}

We call a map like $f$ an \textbf{evolved infimum} and a map like $g$ an \textbf{evolved supremum}. When there is no possible confusion about which are these functions we denote them by $f:=\evinf$ and $g:=\evsup.$
\end{definicion}

These concepts will allow us to formalize rigorously the study and research about posets that come from subsets of lattices where the infimum or the supremum which comes from the lattice could have been lost due to the definition of the new poset. This is the case of the lattice of ideals or subalgebras of an evolution algebra when we pass from considering all the ideals or subalgebras to consider only evolution ideals or evolution subalgebras, respectively. In fact, it is easy to see that there are similar problems in both directions. In the posets of evolution ideals and evolution subalgebras, the infimum of two elements does not necessarily exist and, if it exists, it does not necessarily coincide with the infimum within the bigger lattices of ideals and subalgebras, respectively. And, also in the other direction, in the posets of evolution ideals and evolution subalgebras, the supremum of two elements does not necessarily exist and, if it exists, it does not necessarily coincide with the supremum within the bigger lattices of ideals and subalgebras, respectively.

We introduce a terminology for the pairs were we lost the essential elements of the lattice (the supremum and the infimum) when we work within the \textit{evolution} posets associated to an evolution algebra. Identifying these pairs could be a useful way to study the structure of an evolution algebra.

\begin{definicion}
Let $(L,\leq_{L},\inf,\sup)$ be a lattice and $(E,\leq_{E})$ a subbrset of $(L,\leq_{L}),$ i.e., where $E\subseteq L$ and $\leq_{E}$ is the restriction of $\leq_{L}$ to $E.$ We say that a pair $a,b\in E$ is a \textbf{lattice breakup} of $E$ as a subbrset of $L$ if $\sup(\{a,b\})\notin E.$ We say that a pair $a,b\in E$ is a \textbf{lattice breakdown} of $E$ as a subbrset of $L$ if $\inf(\{a,b\})\notin E.$
\end{definicion}

It is easy to appreciate how our previous definitions can branch into many different directions and paths. For this reason, our next work shows us how to relate the two previous concepts to build meaningful tools in our research on evolution algebras. First of all, we unveil how we can construct proper lattices out of the (merely) posets with which the \textit{evolution} versions of our subobjects provide us. At the same time, these constructions supply us with some good examples of the way these concepts are developed in more concrete objects and manifestations closer to the frame of evolution algebras. This help us to understand better their behaviour and to appreciate their possible utility in other fields.

Firstly, we set up the language necessary to make easy to see that we are effectively before an example of an evlattice whenever we work with the poset of evolution ideals or the poset of evolution subalgebras of an evolution algebra. These evlattices are obviously built over the only two possible lattices that come to our mind when we think about these posets, that is, the lattice of ideals or the lattice of subalgebras, respectively, of the evolution algebra considered viewed as a mere algebra. We establish this rigorously through the next definitions.

Now, we introduce the concepts of being generated by a subset. In the case of algebras, as the posets of its substructures are lattices, the existence of infima and suprema with easy expressions in terms of intersections or sums, respectively, made easy to construct the substructures generated by some special subsets, as sums, unions or intersections of other substructures of the same type which are particularly important in the development of the theory. When we are within the frame of evolution substructures of algebras these characterizations are not available and the work is harder. This makes also harder to understand and less natural in a first view the definitions.

Now first we note that it is possible to speak about \textit{evolution} substructures in a general algebra. That is, it is clear that we do not need to be inside an evolution algebra to define an evolution ideal or evolution subalgebra of this algebra. It is clear that it is enough to ask these substructures to have at least one natural basis. Thus, we extend the definitions of evolution ideals and evolution subalgebras given before to the framework of general algebras in the next definition.

\begin{definicion}\label{17}
Let $A$ be an algebra. We say that an ideal of $A$ or a subalgebra of $A$ is an \textbf{evolution ideal} of $A$ or an \textbf{evolution subalgebra} of $A$ if it has a natural basis.
\end{definicion}

In the next definition, we introduce the concept of being generated \textit{over} a subset. Now the prepositions are important because as we shall see, in contrary to the case of mere substructures of algebras, where the concepts of being generated over or below some specific sets are easy to understand and immediate, this is not the case within the frame of evolution substructures of algebras. Therefore, this situation requires a more careful research which leads us to study in more detail the way some sets generate our substructures and in which sense this generation takes place. We follow first the \textit{over} construction and, later, we will describe the \textit{below} construction in a very similar way.

\begin{definicion}\label{Q}
Let $A$ be an algebra over a field $\mathbb{K}$ and $S$ a subset of $A.$ An \textbf{evolution (two-sided) ideal generated over $S$} is a minimal evolution (two-sided) ideal $B$ of $A$ containing $S$ as a subset (that is, $S\subseteq B\subseteq A$); we denote the set of all minimal evolution ideals of $A$ containing $S$ as $\MinEvid(S),$ and, if this set has only one element, we call this element $\minevid(S).$ Observe that $\minevid(S)=\bigcap_{S\subseteq C, C\vartriangleleft_{\evid}A}C.$ The set of \textbf{evolution ideals generated over $S$}, $\Evido(S),$ is the set of all evolution ideals containing $S.$
\end{definicion}

As we can appreciate at a first sight, whenever the element $\minevid(S)$ exists for every $S\subseteq A$, we are in front of a very special kind of algebras; an algebra whose posets of evolution ideals have a really good behaviour with respect to the infima. The posets of evolution ideals of these algebras behave in a similar way with respect to the infimum-semilattice structure of the respective poset of ideals of an algebra. As these are special, we choose a name for them.

\begin{definicion}
Let $A$ be an algebra over a field $\mathbb{K}$ and $S$ a subset of $A$. We say that $A$ is an \textbf{infimum-semilatticed algebra for evolution ideals} if, for every $S\subseteq A,$ there exists $\minevid(S).$
\end{definicion}

We remember the following notation.

\begin{notacion}
Let $A$ be an algebra and $\mathcal{L}_{\evid}^{A}$ the poset of evolution ideals of $A$ when we consider the order induced by the inclusion. We denote the order induced by the inclusion in this set as $\leq_{\evid}.$
\end{notacion}

We have introduced the terminology to describe the behaviour of the order structure in the lattices of evolution substructures of an algebra with respect to one of the directions of the order given by $\leq_{\evid}$ in the corresponding poset of evolution substructures. Furthermore, we have set some terminology which allows us to describe some kind of relation between the \textit{lower bounds} parts of these order structures of ideals. It is immediately evident that those names suggest to us the way to follow in the next definition, which is the \textit{upper bounds} counterpart of our last definition.

Now, we introduce the concept of being generated below a subset. In the case of ideals of algebras, as the posets of its ideals are lattices and the existence of suprema with easy expressions in terms of sums is guaranteed, this concept is never introduced because is not very useful. However, when we are within the frame of evolution ideals of algebras, these characterizations are not available and we need to study thoroughly what is the behaviour of the construction of ideals below a determined set. It is also likely that a more comprehensive study of these constructions can show us new nuances in the study of mere algebras and lead to a development of a true generation theory within algebras which could be able to study these generation sets and methods in a more exhaustive way.

\begin{definicion}\label{21}
Let $A$ be an algebra over a field $\mathbb{K}$ and $S$ a subset of $A.$ An \textbf{evolution (two-sided) ideal generated below $S$} is a maximal evolution (two-sided) ideal $B$ of $A$ contained in $S$ as a subset (that is, $B\subseteq S\subseteq A$); we denote the set of all maximal evolution ideals of $A$ contained in $S$ as $\MaxEvid(S),$ and, if this set has only one element, we call this element $\maxevid(S).$ Observe that $\maxevid(S)=\bigcup_{C\subseteq S, C\vartriangleleft_{\evid}A}C.$ The set of \textbf{evolution ideals generated below $S$}, $\Evidb(S),$ is the set of all evolution ideals contained in $S.$
\end{definicion}

Again, whenever the element $\maxevid(S)$ exists for every $S\subseteq A$, we are in front of a very special kind of algebras; an algebra whose posets of evolution ideals have a really good behaviour with respect to the suprema. The posets of evolution ideals of these algebras behave in a similar way with respect to the supremum-semilattice structure of the respective posets of ideals of an algebra. As these are special, we choose a name for them.

\begin{definicion}
Let $A$ be an algebra over a field $\mathbb{K}$ and $S$ a subset of $A.$ We say that $A$ is a \textbf{supremum-semilatticed algebra for evolution ideals} if, for every $S\subseteq A,$ there exists $\maxevid(S).$
\end{definicion}

Finally, we put together these properties in both sides to obtain the next definition.

\begin{definicion}
Let $A$ be an algebra over a field $\mathbb{K}$ and $S$ a subset of $A.$ We say that $A$ is a \textbf{latticed algebra for evolution ideals} if, for every $S\subseteq A,$ there exist $\minevid(S)$ and $\maxevid(S).$
\end{definicion}

We let here as an open question the possibility of characterizing and classifying the supremum-semilatticed (infimum-semilatticed, latticed) algebra for evolution ideals. As a side question, it could also be interesting to consider and study algebras verifying other configurations of appearance of evolution ideals in its lattice of ideals. Also, we could consider, in addition, the posets of subalgebras and perform the same study with similar operations in this lattice to pass afterwards to play with more interrelations between these lattice structures and the corresponding posets as we did with the definitions above.

Here, the concepts introduced above will be enough for our purposes. We introduce and remember some notation.

\begin{notacion}
Let $A$ be an algebra. We call $\lattid(A):=\mathcal{L}_{\id}^{A}(\{0\})$ to the lattice of ideals of $A$ (containing $0$) with the operations given by $\sup(\{X,Y\}):=\id(X\cup Y)$ and $\inf(\{X,Y\}):=\id(X\cap Y)=X\cap Y,$ as usual and as it is known. Furthermore, we call $\evlattid(A):=\mathcal{L}_{\evid}^{A}(\{0\})$ to the subposet of $\lattid(A)$ which contains only the evolution ideals.
\end{notacion}

Thus, we have immediately the next theorem which gives us our first example of an evlattice.

\begin{teorema}\label{PE}
Let $A$ be an evolution algebra. The quadruple $$(\evlattid(A),\subseteq,\evinf,\evsup),$$ where $$\evsup(\{X,Y\}):=\MinEvid(X\cup Y) \mbox{\ and\ }\evinf(\{X,Y\}):=\MaxEvid(X\cap Y),$$ is an evlattice over the lattice $(\lattid(A),\subseteq,\inf,\sup).$
\end{teorema}

\begin{proof}
First, we see that $(\evlattid(A),\subseteq,\evinf)$ is an infevlattice over the infimum-lattice $(\lattid(A),\subseteq,\inf).$ We take an arbitrary two-element subset $\{X,Y\}\subseteq\evlattid(A)\subseteq\lattid(A).$ We remember that $\inf(\{X,Y\}):=X\cap Y.$ As $\evinf(\{X,Y\}):=\MaxEvid(X\cap Y)$ and remembering the definition of $\MaxEvid(X\cap Y),$ we know that the elements in $\evinf(\{X,Y\})$ are the maximal evolution ideals of $A$ (in the sense that there are no other evolution ideals strictly containing them and contained in $X\cap Y$ at the same time) contained in the set $X\cap Y.$ As $\{0\}$ is an evolution ideal contained in $X\cap Y$ and the algebra has finite dimension, there should exist maximal evolution ideals (in the sense that there are no other evolution ideals strictly containing them and contained in $X\cap Y$ at the same time) contained in $X\cap Y.$ Thus, choose arbitrarily one of these evolution ideals $Z\in\MaxEvid(X\cap Y).$ Then, as said before, that, particularly, means that $Z\subseteq X\cap Y=\inf(\{X,Y\}).$ Hence, as we chose $Z$ arbitrarily, we have just proved that every $Z\in\evinf(\{X,Y\})$ verifies $Z\subseteq\inf(\{X,Y\}).$ Therefore, as, at the beginning the two-element subset $\{X,Y\}\subseteq\evlattid(A)\subseteq\lattid(A)$ was arbitrarily chosen, we have just conclude that $(\evlattid(A),\subseteq,\evinf)$ is an infevlattice over the infimum-lattice $(\lattid(A),\subseteq,\inf).$

Seeing that $(\evlattid(A),\subseteq,\evsup)$ is a supevlattice over the supremum-lattice $(\lattid(A),\subseteq,\sup)$ is analogous and dual. Thence, it is proved that the quadruple $(\evlattid(A),\subseteq,\evinf,\evsup)$ is an evlattice over the lattice $(\lattid(A),\subseteq,\inf,\sup)$ when $A$ is an evolution algebra, as we wanted. This proof is concluded.
\end{proof}

In the previous theorem we could have eluded the hypothesis of $A$ being an evolution algebra and ask $A$ just to be an algebra. This would just have introduced the small problem in the proof that in some cases $\evsup(\{X,Y\})$ could be empty (in contrary to what happen when we ask $A$ to be an evolution algebra) but the rest of the proof would work equally. Thus, we have the next theorem, which is more general, just refining minimally our previous proof.

\begin{teorema}\label{PEA}
Let $A$ be an algebra. The quadruple $$(\evlattid(A),\subseteq,\evinf,\evsup),$$ where $$\evsup(\{X,Y\}):=\MinEvid(X\cup Y) \mbox{\ and\ }\evinf(\{X,Y\}):=\MaxEvid(X\cap Y),$$ is an evlattice over the lattice $(\lattid(A),\subseteq,\inf,\sup).$
\end{teorema}

Now that we have set that our first examples of evlattices are deeply connected with the lattices of ideals and the posets of evolution ideals of algebras, it is natural to ask ourselves how can we recover some of the power that lattice theory has in the study of algebras for these new structures. For this reason, we see now how we can construct lattices from evlattices. These lattices will be subposets (sublattices) of the basis lattices in the construction of the respective evlattices. First we establish the way to construct a lattice out of an evlattice.

\begin{definicion}\label{27}
Let $(L,\leq_{L},\inf,\sup)$ be a lattice and consider its subbrset $(E,\leq_{E}),$ with $E\subseteq L$ and $\leq_{E}$ the restriction of $\leq_{L}$ to $E.$

\begin{enumerate}
    \item Let $(E,\leq_{E},f)$ be an infevlattice over the infimum-lattice $(L,\leq_{L},\inf)$ and suppose that there exist a subset $F\subseteq E$ and a map $f'\colon\binom{F}{2}\to F$ verifying that $f'(\{a,b\})\in f(\{a,b\}),$ for every $a,b\in F$ with $a\neq b,$ such that $(F,\leq_{F},f')$ is an infimum-lattice with the infimum given by $f'.$ Then we say that $(F,\leq_{F},f')$ is an infimum-lattice \textbf{definable within (the infimum-lattice frame determined by the infevlattice)} $(E,\leq_{E},f).$
    \item Let $(E,\leq_{E},g)$ be a supevlattice over the supremum-lattice $(L,\leq_{L},\sup)$ and suppose that there exist a subset $G\subseteq E$ and a map $g'\colon\binom{G}{2}\to G$ verifying that $g'(\{a,b\})\in g(\{a,b\}),$ for every $a,b\in G$ with $a\neq b,$ such that $(G,\leq_{G},g')$ is a supremum-lattice with the supremum given by $g'.$ Then we say that $(G,\leq_{G},g')$ is a supremum-lattice \textbf{definable within (the supremum-lattice frame determined by the supevlattice)} $(E,\leq_{E},g).$
    \item Let $(E,\leq_{E},f,g)$ be an evlattice over the lattice $(L,\leq_{L},\inf,\sup)$ and suppose that there exist a subset $H\subseteq E$ and maps $f',g'\colon\binom{H}{2}\to H$ verifying that $f'(\{a,b\})\in f(\{a,b\})$ and $g'(\{a,b\})\in g(\{a,b\})$ for every $a,b\in H$ with $a\neq b,$ such that $(H,\leq_{H},f',g')$ is a lattice with the infimum given by $f'$ and the supremum given by $g'.$ Then we say that $(H,\leq_{H},f',g')$ is a lattice \textbf{definable within (the lattice frame determined by the evlattice)} $(E,\leq_{E},f,g).$
    \end{enumerate}
\end{definicion}

Now, we translate this to our particular case.

\begin{definicion}\label{28}
Let $A$ be an algebra over a field $\mathbb{K}$ and $S$ a subset of $A$. Define
$$
\evinf(\{X,Y\}):=\MaxEvid (X\cap Y)\mbox{\ and\ } \evsup(\{X,Y\}):=\MinEvid(X\cup Y).
$$
    
    \begin{enumerate}
        \item We say that the set of lattices definable within the evlattice $$(\Evido(S),\subseteq, \evinf, \evsup)$$ is the \textbf{set of lattices definable over $S$} and we denote it by $\OverEvid(S).$
        \item We say that the set of lattices definable within the evlattice $$(\Evidb(S),\subseteq, \evinf, \evsup).$$ is the \textbf{set of lattices definable below $S$} and we denote it by $\BelowEvid(S).$
    \end{enumerate}
\end{definicion}

When we speak about the extrema we are speaking about the infima and the suprema indistinctly. Now we are going to order lattices.

\begin{definicion}\label{29}
Let $(B,\leq)$ be a brset and $(L,\leq_{L},\inf_{L},\sup_{L})$ and $(M,\leq_{M},\inf_{M},\sup_{M})$ be lattices with $L,M\subseteq B$ and $\leq_{L}, \leq_{M}$ the restrictions of $\leq$ to $L$ and $M,$ respectively. We say that $L\leq_{\latt}M$ if $L\subseteq M$ and, for every subset of $L$ of two different elements $\{a,b\}\subseteq L,$ we have that $\inf_{L}(\{a,b\})=\inf_{M}(\{a,b\})$ and $\sup_{L}(\{a,b\})=\sup_{M}(\{a,b\});$ that is, if $L$ is a subset of $M$ and the extrema are respected in the extension from $L$ to $M.$
\end{definicion}

It is easy to see that $\leq_{\latt}$ is an order relation between sublattices of a brset. We call this order the \textbf{order induced by the inclusion between the lattices respecting the extrema}. Now we introduce a technical result which follows easily from the Zorn's Lemma.

\begin{teorema}\label{30}
Let $A$ be an algebra over a field $\mathbb{K}$ and $S$ a subset of $A$. Every lattice definable over (below) $S$ is contained in a maximal lattice definable over (below) $S$ in the sense given by the order induced by the inclusion between the lattices respecting the extrema $\leq_{\latt}.$
\end{teorema}

\begin{proof}
We use the Zorn's Lemma. In fact, take a chain $\mathfrak{l}$ of lattices definable over (below) $S$ in the sense of $\leq_{\latt}.$ Then $\bigcup_{\mathcal{L}\in\mathfrak{l}}\mathcal{L}$ with the order induced by the inclusion between the lattices respecting the extrema is a lattice definable over (below) $S$ and, therefore, it is an upper bound for the chain of lattices definable over (below) $S$ given by $\mathfrak{l}$ in the partial order given by the inclusion respecting the extrema in the set of lattices definable over (below) $S.$ Now, the Zorn's Lemma say that there should exist maximal lattices definable over (below) $S$ in the sense given by the order induced by the inclusion between the lattices respecting the extrema, as we wanted to prove.
\end{proof}

The subset $S$ is important to characterize the algebra $A.$

\begin{definicion}\label{31}
Let $A$ be an algebra over a field $\mathbb{K}$ and $S$ a subset of $A$. We say that $S$ is a \textbf{cut of over unicity for} $A$ if there exist only one maximal lattice definable over $S.$ Similarly, we say that $S$ is a \textbf{cut of below unicity for} $A$ if there exist only one maximal lattice definable below $S.$ 
\end{definicion}

We require some notation to represent \textit{tightness} in brsets.

\begin{definicion}
Let $(L,\leq_{L})$ be a brset. We say that a chain of inequalities $a\leq_{L}b_{1},b_{2}\dots,b_{n-1},b_{n}\leq_{L}c$ \textbf{admits a tightening} if there exist other elements $l,m\in L$ with $l\neq a$ or $m\neq c$ verifying $a\leq_{L}l\leq_{L}b_{1},b_{2}\dots,b_{n-1},b_{n}\leq_{L}m\leq_{L}c.$ A chain $a\leq_{L}b_{1},b_{2}\dots,b_{n-1},b_{n}\leq_{L}c$ not admitting a tightening is called \textbf{tight} and we represent it as $a\leq^{\between}_{L}b_{1},b_{2}\dots,b_{n-1},b_{n}\leq^{\between}_{L}c.$
\end{definicion}

Now we have the tools that we need to prove the next theorem valid for evolution algebras.

\begin{teorema}\label{33}
Let $A$ be an evolution algebra over a field $\mathbb{K}$ and $S$ a subset of $A.$ Then we have the following.
\begin{enumerate}
    \item If $S$ is a cut of over unicity for $A,$ then every pair of two different evolution ideals $I,J$ of $A$ containing $S$ such that there exist evolution ideals $R$ with $S\subseteq R\in\MaxEvid(I\cap J)$ verifies that there exists just one evolution ideal $T$ of $A$ such that $T\in\MinEvid(I\cup J).$ Moreover, under this circumstance, the evolution ideal $R$ is unique; that is, if there exists an evolution ideal $R$ such that $S\subseteq R\in\MaxEvid(I\cap J),$ then this is the only one.
     \item If $S$ is a cut of below unicity for $A,$ then every pair of two different evolution ideals $I,J$ of $A$ contained in $S$ such that there exist evolution ideals $R$ with $S\supseteq R\in\MinEvid(I\cup J)$ verifies that there exists just one evolution ideal $T$ of $A$ such that $T\in\MaxEvid(I\cap J).$ Moreover, under this circumstance, the evolution ideal $R$ is unique; that is, if there exists an evolution ideal $R$ such that $S\supseteq R\in\MinEvid(I\cup J),$ then this is the only one.
\end{enumerate}
\end{teorema}

\begin{proof}

$1.$ Let suppose that there exist another evolution ideal $L$ of $A$ such that $L\in\MinEvid(I\cup J).$ Then, choosing an evolution ideal $S\subseteq R\in\MaxEvid(I\cap J),$ we could define the lattice definable over $S$ given by $S\subseteq R\leq^{\between}_{\evid}I,J\leq^{\between}_{\evid}L,$ which would avoid the maximal extension of the lattice definable over $S$ given by $S\subseteq R\leq^{\between}_{\evid}I,J\leq^{\between}_{\evid}T$ to be the only one maximal lattice definable over $S.$ Now, the unicity of $R$ follows similarly from the unicity of the maximal lattice.

$2.$ Let suppose that there exist another evolution ideal $L$ of $A$ such that $L\in\MaxEvid(I\cap J).$ Then, choosing an evolution ideal $S\supseteq R\in\MinEvid({I\cup J}),$ we could define the lattice definable below $S$ given by $L\leq^{\between}_{\evid}I,J\leq^{\between}_{\evid}R\subseteq S,$ which would avoid the maximal extension of the lattice definable below $S$ given by $T\leq^{\between}_{\evid}I,J\leq^{\between}_{\evid}R\subseteq S$ to be the only one maximal lattice definable below $S.$ Again, the unicity of $R$ follows similarly from the unicity of the maximal lattice.
\end{proof}

\begin{remark}
We note that in the proof of this last result we are using the fact of $A$ being an evolution algebra just in the first assertion to ensure the existence of $T$ in all situations. This is similar to what we commented that was happening in the discussion exposed before introducing the Theorem \ref{PEA} after proving its \textit{evolution} pair Theorem \ref{PE} introduced before and having both similar proofs. Coming back to this case, again, all the strength of $A$ being an evolution algebra is not needed and it would be enough to ask the algebra $A$ to have exactly one maximal evolution ideal $M,$ which is what happen to evolution algebras: being themselves their maximal evolution ideal. That is, it would be enough to ask $A$ to have a poset of evolution ideals being bounded for above by just one evolution ideal $M.$ Nevertheless, then we could decide to forget about $A$ and concentrate just in this maximal evolution ideal $M$ of $A$ taking it as an algebra as our algebra under research. Hence, as $M$ would be an evolution algebra, then we would recover the same assertion but for this ideal $M$ instead of $A,$ i.e., writing $M$ where we wrote $A$ in the previous result, if we assume that $S\subseteq M.$ Otherwise, if $S\nsubseteq M,$ we would have that $\Evido(S)=\emptyset$ and, therefore, the only lattice definable over $S$ would be the empty lattice (we accept the usual convention of considering it as a lattice). This discussion does not affect the second assertion of the theorem (at least as long as we consider the ideal $\{0\}$ to be an evolution ideal with natural basis $B=\emptyset$).
\end{remark}

Also, as our algebras are finite-dimensional, we can restate the previous theorem before closing this section.

\begin{teorema}\label{35}
Let $A$ be an evolution algebra over a field $\mathbb{K}$ and $S$ a subset of $A.$ Then we have the following.
\begin{enumerate}
    \item If $S$ is a cut of over unicity for $A,$ then every pair of two different evolution ideals $I,J$ of $A$ containing $S$ such that there exist evolution ideals $R$ with $S\subseteq R\subseteq I\cap J$ verifies that there exists just one evolution ideal $T$ of $A$ such that $T\in\MinEvid(I\cup J).$ Moreover, there exist a maximal $R$ and it is unique.
     \item If $S$ is a cut of below unicity for $A,$ then every pair of two different evolution ideals $I,J$ of $A$ contained in $S$ such that there exist evolution ideals $R$ with $S\supseteq R\supseteq I\cup J$ verifies that there exists just one evolution ideal $T$ of $A$ such that $T\in\MaxEvid(I\cap J).$ Moreover, there exist a maximal $R$ and it is unique.
\end{enumerate}
\end{teorema}

\section{Minimal ideals}

From now on, when we speak about a minimal ideal $I$ of an algebra $A$ we are referring to those ideals which are non-zero and verify that, for every ideal $J\neq\{0\}$ of the algebra $A,$ the inequality $J\subseteq I$ implies $J=I.$ However, a minimal evolution ideal of an algebra $A$ is a non-zero evolution ideal $I$ of the algebra $A$ verifying that, for every \textit{evolution ideal} $J\neq\{0\}$ of the algebra $A$ the inequality $J\subseteq I$ implies $J=I.$ Thus, we have to take special care because saying that $I$ is a minimal evolution ideal of the algebra $A$ is not equivalent to saying that $I$ is a minimal ideal of the algebra $A$ and an evolution ideal of the algebra $A$: the second fact implies the first always, but the reciprocal is not always true.

\begin{definicion}
Let $B:=\{e_{i}\}_{i\in\Lambda}$ be a basis of a vector space $A.$ We define the \textbf{projections} over the basic elements in $B$ such that, for every element $a\in A,$ we set $\proj_{B,e_{j}}(a)$ as the coefficient of $e_{j}$ in the unique representation of $a$ as a linear combination of the elements in the basis $B.$
\end{definicion}

\begin{definicion}\label{37}
Let $A$ be an algebra. We say that $A$ is \textbf{simple} if $A^{2}\neq \{0\}$ and $\{0\}$ is the only proper ideal of $A.$ We say that $A$ is \textbf{semiprime} if there are no non-zero ideals $I$ of $A$ such that $I^{2}=\{0\}.$ We say that $A$ is \textbf{nondegenerate by the left} if $a(Aa)=\{0\}$ implies $a=0.$ We say that $A$ is \textbf{non-degenerate} if $A$ has a natural basis $B=\{e_{i}\mid i\in \Lambda\}$ such that $e_{i}^{2}\neq0$ for every $i\in\Lambda.$
\end{definicion}

We denote $A^{\overline{2}}:=\{ab\mid a,b\in A\},$ so this set is just the set formed by all the products of elements of $A.$ Carefully, we note and remember here that $\linspan(A^{\overline{2}})$ is usually denoted by $A^{2}.$ Now, we prove that, in every simple algebra $A,$ we have that $\linspan(A^{\overline{2}})=A.$

\begin{proposicion}
Let $A$ be a simple algebra. Then $\linspan(A^{\overline{2}})=A.$
\end{proposicion}

\begin{proof}
Obviously, $\linspan(A^{\overline{2}})\subseteq A.$ Now we prove in the other direction that $A\subseteq\linspan(A^{\overline{2}}).$ Let $a,b\in A$ such that $ab\neq 0$ and let $c\in A$ be an arbitrary element. Then $c\in A=\id(\{ab\}).$ Now, $l(ab)$ and $(ab)r$ are in $A^{\overline{2}}$ for arbitrary $l,r\in A$ so 
$c\in A=\id(\{ab\})\subseteq\linspan(A^{\overline{2}}).$ This concludes our proof.
\end{proof}

Through this text, when we say \textit{non-associative}, we mean non-necessarily-associative, as usual. When $A$ is associative, semiprimeness and nondegeneracy are equivalent.

\begin{proposicion}\cite[Section 9.2]{rtana}
Let $A$ be an associative algebra. Then $A$ is semiprime if and only if $A$ is nondegenerate.
\end{proposicion}

But this last proposition is not true in the general nonassociative case, as showed (specifically for evolution algebras) in \cite[Remark 2.31]{cvs}. Also we have the following result.

\begin{proposicion}\label{A}
Let $A$ be a nondegenerate evolution algebra over a field $\mathbb{K}$ with a natural basis $B=\{e_{i}\}_{i\in\Lambda}$ and the multiplication given in this basis by $e_{i}e_{i}=\sum_{j=1}^{n_{i}}a_{ij}e_{j}$ for every $i\in\Lambda$ and for some naturals $n_{i}\in\mathbb{N}.$ Then $a_{ii}\neq 0$ for all $i\in\Lambda.$
\end{proposicion}

\begin{proof}
As $A$ is nondegenerate, we have that $e_{i}(Ae_{i})\neq\{0\}$ for every $i\in\Lambda.$ Suppose that there exists $j\in\Lambda$ with $a_{jj}=0.$ Then, for every $a\in A,$ we have that $\proj_{B,e_{j}}(ae_{j})=0$ and, thus, $e_{j}(Ae_{j})=\{0\},$ a contradiction.
\end{proof}

In fact, in the case of evolution algebras we have the following proposition.

\begin{proposicion}\cite[Proposition 2.30]{cvs}\label{41}
Let $A$ be an evolution algebra with non-zero product and consider the following conditions.

\begin{enumerate}
    \item The algebra $A$ is nondegenerate.
    \item The algebra $A$ is semiprime.
    \item The algebra $A$ has no non-zero evolution ideals $I$ verifying $I^{2}=\{0\}.$
    \item The algebra $A$ is non-degenerate.
\end{enumerate}

Then these conditions relate logically as follows: $1\Rightarrow2\Leftrightarrow3\Rightarrow4.$
\end{proposicion}

It is important to mention here \cite[Remark 2.31]{cvs}, where the authors show that $2\nRightarrow1$ and $4\nRightarrow3.$

From now on we will work mainly with the strongest of the four previous conditions (that is, nondegeneracy) because, even under this ideal situation, things will be complicated. Before introducing the next definition, we clarify that, during all this text, when we speak about an \textit{idempotent} we always mean an \textit{idempotent element}.

\begin{definicion}\label{42}
Let $A$ be an algebra and $e\in A$ be an idempotent (i.e., $e^{2}=e$). We say that $e$ is a \textbf{minimal idempotent} of $A$ if $\id(\{e\})$ is a minimal ideal. In general, we say that an arbitrary element $a\in A$ is a \textbf{minimal element} of $A$ if $\id(\{a\})$ is a minimal ideal of $A.$
\end{definicion}

We note that, following the previous definition, a minimal idempotent is never $0$ because we established before that a minimal ideal has to be different from zero. In fact, this notion could be extended and we could introduce an order between all the elements of $A$ (and, particularly, the idempotents) induced by the order in its lattice of ideals considering the ideals generated by these elements, as it is done or suggested in the discussions after the definitions \cite[Definition 6.11]{sk} or \cite[Definition 8.2.1]{twp}.

We know the following theorem for associative rings. Our intention is to find something similar for evolution algebras.

\begin{teorema}\label{F}\cite[Corollary II.1.6]{mp}\cite[Discussion between Theorem A.8 and Theorem A.9]{mmw}
Let $R$ be a semiprime associative ring. Then every minimal left ideal $I$ of $R$ verifies that there exists an idempotent $e\in I$ such that $I=Re=\id(\{e\}).$
\end{teorema}

Unfortunately, this is not the case in general for a nondegenerate evolution algebra. It is even possible to find counterexamples among simple nondegenerate evolution algebras. We see this with an example of an evolution algebra over the finite field $\mathbb{Z}_{3}.$

\begin{ejemplo}\label{Exa}
First of all, we choose an evolution algebra $A$ over $\mathbb{Z}_{3}$ with a natural basis $B=\{e_{1},e_{2}\}$ and multiplication given by $e_{1}e_{1}=ae_{1}+be_{2},$ $e_{2}e_{2}=ce_{1}+de_{2},$ where $a,b,c,d\in\mathbb{Z}_{3}.$ We have to determine $a,b,c,d\in\mathbb{Z}_{3}$ in order to obtain a nondegenerate evolution algebra verifying that $I=A=\linspan(\{e_{1},e_{2}\})$ is a minimal ideal and it is not generated by an idempotent. The first condition that we can see is that, in order to have a nondegenerate evolution algebra and using Proposition \ref{A}, we must have $d\neq0$ and $a\neq 0.$

Moreover, if we want $I$ to be a minimal ideal, then it is necessary that $b\neq0$ and $c\neq0$ because, otherwise, $\id(\{e_{1}\})$ or $\id(\{e_{2}\})$ would be ideals strictly contained in $I,$ which contradicts $I$ being a minimal ideal. In fact, under the hypothesis of nondegeneracy, $A^{2}\neq\{0\}$ and, therefore, $I=A$ is a minimal ideal means that $A$ is a simple algebra. Furthermore, under the hypothesis $d\neq0$ and $a\neq 0,$ we have that $A$ is simple if and only if $b\neq 0, c\neq0$ and $\frac{a}{c}\neq\frac{b}{d}.$ We will see this last statement in more detail.

In one direction, if $A$ is simple, then $b\neq0\neq c$ because otherwise there would be ideals strictly contained in $A.$ Also, if $\frac{a}{c}=\frac{b}{d},$ then $\id(\{ae_{1}+be_{2}\})=\linspan(\{ae_{1}+be_{2}\})\subsetneq A$ and $A$ would not be simple. In the other direction, suppose $b\neq0\neq c$ and $\frac{a}{c}\neq\frac{b}{d}.$ Then, given $0\neq q\in A$ arbitrary, we have $0\neq q=ge_{1}+he_{2}.$ Suppose $g\neq0\neq h.$ Then $qe_{1}=g(e_{1}e_{1})=gae_{1}+gbe_{2}$ so $ae_{1}+be_{2}\in\id(\{q\})$ and $qe_{2}=h(e_{2}e_{2})=hce_{1}+hde_{2}$ so $ce_{1}+de_{2}\in\id(\{q\}).$ Thus, as $b\neq0\neq c$ and $\frac{a}{c}\neq\frac{b}{d},$ we obtain that $\id(\{q\})=A.$ Suppose now, for example, $h=0.$ Then $q=ge_{1}$ and $qe_{1}=g(e_{1}e_{1})=gae_{1}+gbe_{2}$ so $ae_{1}+be_{2}\in\id(\{q\}).$ As $b\neq 0,$ this implies $e_{1},e_{2}\in\id(\{q\})$ and it is the end of our proof.

Now, under the hypothesis of $A$ being simple, we have that $A$ is nondegenerate if and only if $d\neq0$ and $a\neq 0.$ 

We prove now this last statement. First of all, we remember that the hypothesis of $A$ being simple implies directly that $b\neq0\neq c$ and $\frac{a}{c}\neq\frac{b}{d}.$ One of the implications is clear as a consequence of Proposition \ref{A}. In the other direction, suppose $d\neq0$ and $a\neq 0$ and fix $0\neq q\in A$ arbitrary. Then $q=ge_{1}+he_{2}.$ So we have $e_{1}q=g(e_{1}e_{1})=g(ae_{1}+be_{2})=gae_{1}+gbe_{2}$ and $e_{2}q=h(e_{2}e_{2})=h(ce_{1}+de_{2})=hce_{1}+hde_{2}.$ So $q(e_{1}q)=q(gae_{1}+gbe_{2})=(ge_{1}+he_{2})(gae_{1}+gbe_{2})=g^{2}a(e_{1}e_{1})+hgb(e_{2}e_{2})=g^{2}a(ae_{1}+be_{2})+hgb(ce_{1}+de_{2})=(g^{2}a^{2}+hgbc)e_{1}+(g^{2}ab+hgbd)e_{2}$ and $q(e_{2}q)=q(hce_{1}+hde_{2})=(ge_{1}+he_{2})(hce_{1}+hde_{2})=ghc(e_{1}e_{1})+h^{2}d(e_{2}e_{2})=ghc(ae_{1}+be_{2})+h^{2}d(ce_{1}+de_{2})=(ghca+h^{2}dc)e_{1}+(ghcb+h^{2}d^{2})e_{2}.$ Now, suppose $(ghca+h^{2}dc)e_{1}+(ghcb+h^{2}d^{2})e_{2}=0=(g^{2}a^{2}+hgbc)e_{1}+(g^{2}ab+hgbd)e_{2}.$ As $d\neq 0\neq a$ and we are in $\mathbb{Z}_{3},$ this implies $g^{2}+hgbc=g^{2}a^{2}+hgbc=0=ghcb+h^{2}d^{2}=ghcb+h^{2}.$ Thus, we have $g^{2}=h^{2}$ so there are just two possibilities: $g=h=0$ or $g\neq0\neq h;$ but, as $q\neq 0,$ we get $g\neq0\neq h.$ This implies $ghcb+h^{2}d^{2}=gcb+h=0$ so $h=-gcb.$ Similarly, as $g^{2}a^{2}+hgbc=g+hbc=0,$ we obtain $g=-hbc.$ Thus, $0=ghca+h^{2}dc=-h^{2}c^{2}ba+h^{2}dc,$ which implies $c^{2}ba=dc$ so $cba=d$ and, finally, $cb=da,$ which is a contradiction with the fact that $\frac{a}{c}\neq\frac{b}{d}.$ This concludes the proof.

So, if we find an algebra $A$ such that $0\notin\{a,b,c,d\}$ and $\frac{a}{c}\neq\frac{b}{d},$ then this algebra $A$ is simple and nondegenerate. We examine the possibility of the existence of non-zero idempotents under these conditions (simplicity and nondegeneracy). A non-zero idempotent in $I=A$ must be an element $0\neq ke_{1}+le_{2},$ with $l,k\in\mathbb{Z}_{3}$ such that $(ke_{1}+le_{2})^{2}=(ke_{1}+le_{2})(ke_{1}+le_{2})=k^{2}(ae_{1}+be_{2})+l^{2}(ce_{1}+de_{2})=(cl^{2}+ak^{2})e_{1}+(dl^{2}+bk^{2})e_{2}=ke_{1}+le_{2}.$ This means that we need to find a solution $(k,l)\neq(0,0)$ to the system of equations

\begin{equation}
\left\{
\begin{aligned}
k &= cl^{2}+ak^{2}\\
l &= dl^{2}+bk^{2}.
\end{aligned}
\right.
\end{equation}

To find a counterexample, it is enough to exhibit one system of the family of systems that appear when $a,b,c,d$ vary in $\mathbb{Z}_{3}\smallsetminus\{0\}$ verifying $\frac{a}{c}\neq\frac{b}{d}$ such that its only solution is $(k,l)=(0,0).$ The easiest counterexample arises when we set $a=c=d=1$ and $b=2.$ We verify it trying all the possible solutions in the finite field $\mathbb{Z}_{3}.$ As the only solution of the equation 
\begin{equation}
\left\{
\begin{aligned}
k &= l^{2}+k^{2}\\
l &= l^{2}+2k^{2}
\end{aligned}
\right.
\end{equation}
is $(k,l)=(0,0)$ and the choice of the parameters verify the conditions exposed above ($a,b,c,d\notin\{0\}$ and $\frac{a}{c}\neq\frac{b}{d}$), then we have found our counterexample. The evolution algebra $A=\linspan(\{e_{1},e_{2}\}),$ with the multiplication given by $e_{1}e_{1}=e_{1}+2e_{2}, e_{2}e_{2}=e_{1}+e_{2},$ is nondegenerate and simple but it is not generated by any idempotent because there is no idempotent other than $0.$
\end{ejemplo}

Now two questions arise. The first one is that the chosen field in the previous example is very special and we would like to deal with fields closer to natural applications or, at least, to the most common among mathematics, that is, more similar to $\mathbb{R}$ and $\mathbb{C};$ in fact, we think that it is easy to find other counterexamples as the one above for other finite fields using similar techniques. The other question is about forcing the existence of a non-zero idempotent. The previous example does not work just because there is no non-zero idempotent to choose. The second question is easy to answer. Of course, if we force the existence of a non-zero idempotent $e$ in a minimal ideal $I$ of any algebra $A,$ then, obviously, $\id(\{e\})=I.$ The first question is more intricate and it leads us going in depth into the study of the previous example taking a close look at the role developed there by the field and going into detail about what is happening there and what happen when we work with ``nicer" fields. After all, we will see that we are mainly concerned with the existence of non-zero solutions for at least one element of a particular parametric family of systems of equations.

To make the work easier, we pay special attention to simple algebras, as it is the one presented in the example above. Under this hypothesis, our work is simplified just to the search of a non-zero idempotent in the algebra.

\begin{construccion}\label{B}
Let $A$ be a simple nondegenerate evolution algebra of dimension $n$ over a field $\mathbb{K}$ with a natural basis $B=\{e_{i}\mid i\in\{1,\dots,n\}\}$ and the multiplication given by \begin{equation}e_{i}e_{i}=\sum_{j=1}^{n}a_{ij}e_{j}\label{AA}\end{equation} for every $i\in\{1,\dots,n\}.$ Our previous counterexample was a toy one and we had the opportunity to extract easily more information from the coefficients $a_{ij}$ but in this case we cannot go so far. The only thing that it is clear, by Proposition \ref{A}, is that, as $A$ is nondegenerate, then $a_{ii}\neq 0$ for all $i\in\{1,\dots,n\}.$ Now we have to study the possibility of the existence of non-zero idempotents within this algebra. As we already know, if we prove the existence of one non-zero idempotent, then $A$ is generated by this idempotent.

An element $b=\sum_{k=1}^{n} x_{k}e_{k}$ is idempotent if and only if it verifies $$b^{2}=(\sum_{k=1}^{n} x_{k}e_{k})(\sum_{k=1}^{n} x_{k}e_{k})=\sum_{k=1}^{n} x_{k}^{2}e^{2}_{k}=\sum_{k=1}^{n} x_{k}^{2}(\sum_{j=1}^{n}a_{kj}e_{j})=\sum_{j=1}^{n} (\sum_{k=1}^{n}x_{k}^{2}a_{kj})e_{j}=b.$$
That means that $b$ is idempotent if and only if $x_{j}=\sum_{k=1}^{n}x_{k}^{2}a_{kj}$ for all $j\in\{1,\dots,n\}.$ What we obtain is a parametric family of systems of equations whose parameters are precisely the coefficients in the structure matrix of the algebra $A$ relative to the basis $B.$ We just have to study this parametric family of systems of equations. Thus, our problem about the existence of non-zero idempotents in the simple nondegenerate evolution algebra $A$ of dimension $n$ over a field $\mathbb{K}$ reduces to finding a natural basis $B$ of the algebra $A$ such that the coefficients in the structure matrix of the algebra $A$ relative to the basis $B$ generate a system of equations allowing non-zero solutions. This system of equations looks as follows.

For clarity, we write this equation in a more recognizable way.

\begin{equation}\label{E}
\left\{
\begin{aligned}
x_{1} &= a_{11}x_{1}^{2}+\cdots+a_{n1}x_{n}^{2}\\
\vdots\\
x_{n} &= a_{1n}x_{1}^{2}+\cdots+a_{nn}x_{n}^{2},
\end{aligned}
\right.
\end{equation}
with $a_{ij}\neq 0,$ for all $i,j\in\{1,\dots, n\}$ at the same time, subjected to the condition of $A$ (with the product that this structure matrix define in the Equation (\ref{AA}) via the basis $B$) being a simple nondegenerate evolution algebra of dimension $n$ over the field $\mathbb{K}.$ We call this last condition about the coefficients in the system \textbf{condition SN}. When the condition is not asking $A$ to be nondegenerate is called \textbf{condition S} and when the condition is not asking $A$ to be simple is called \textbf{condition N}.

Each one of the equations in the previous system corresponds to a (possibly reducible) quadric algebraic hypersurface in the space $\mathbb{K}^{n}.$
\end{construccion}

Now we have to look in more depth at the systems of equations arising in the previous construction. In the general case and for some arbitrarily chosen parameters, a system of equations like (\ref{E}) does not necessarily have a non-zero solution, as we saw in our previous example. It is also not directly evident whether or not fields such as $\mathbb{Q}, \mathbb{R},$ $\mathbb{C}$ are fields exhibiting non-zero solution for this kind of systems. At a first sight, one could think that the algebraic closedness of $\mathbb{C}$ ensures that these systems will have always a non-zero solution in the complex case (when $\mathbb{K}=\mathbb{C}$). Nevertheless, this is not trivial, because it is easy to find examples of systems of equations similar to (\ref{E}) having no more than the zero solution, for example:
\begin{equation}
\left\{
\begin{aligned}
y &= y^{2}\\
(1-y)x &= (1-y)-1.
\end{aligned}
\right.
\end{equation}

Or, even more similar to (\ref{E}) but without verifying condition S:

\begin{equation}\label{MU}
\left\{
\begin{aligned}
x &= x^{2}-y^{2}\\
y &= x^{2}-y^{2},
\end{aligned}
\right.
\end{equation}
which does not have non-zero solution in any field. In fact, (\ref{MU}) does not verify condition S because the evolution algebra $A$ given by the associated structure matrix has product $e_{1}e_{1}=e_{1}+e_{2}$ and $e_{2}e_{2}=-e_{1}-e_{2},$ which is not simple because $\linspan\{e_{1}+e_{2}\}$ is a proper ideal of $A.$ These kind of examples extend easily to higher dimensions and they are the reason we have to consider the extra conditions S and SN introduced above for the systems under consideration.

In fact, we have to split hairs about the relation between the field and the system even more. We have to put some order in the systems of equations appearing.

\begin{definicion}\label{DDD}
A quadratic expression of the form $x_{1} = a_{1}x_{1}^{2}+\cdots+a_{n}x_{n}^{2},$ as those appearing in the system (\ref{E}) above, is said to be in \textbf{zero-intersect-evolution-axis form}. We fix a field $\mathbb{K}.$ A system of the form
\begin{equation}\label{GE}
\left\{
\begin{aligned}
x_{1} &= p_{1}(x_{1},\dots,x_{n})\\
\vdots\\
x_{n} &= p_{n}(x_{1},\dots,x_{n}),
\end{aligned}
\right.
\end{equation}
where all the polynomials $p_{i}\in\mathbb{K}[X_{1},\dots,X_{n}],$ is said to be \textbf{compatible formally evaluative trivializing} if it is possible to prove that its only solution is the trivial one $(x_{1},\dots,x_{n})=(0,\dots,0)$ using only the sum of the field and evaluations of the polynomials $p_{i}$ over the unknowns $x_{1},\dots,x_{n}$ and the constant $0$ to arrive at evaluations ensuring that the $p_{i}(x_{1},\dots,x_{n})$ are forced to be constantly zero by the system. That is, we forbid us to use most of the field particular structure and arbitrary evaluations of the polynomials over the field $\mathbb{K}$ itself or other variables while solving the system.
\end{definicion}

These systems easy to solve are problematic to some extent for our purposes precisely because they are too easy to solve and this provoke that they present some pathological behaviour when we want to perform a search for fields solving certain kind of systems in a non-trivial way. The main problem with this kind of systems is that they are not going to exhibit new solutions when we consider field extensions of $\mathbb{K}$ because their solution is not dependent in any form from any extension of the field. Thus and moreover, among those systems we will see that we can find systems only admitting the trivial solution whichever the field $\mathbb{K}$ is. We see some examples increasing in difficulty to elucidate better and make clear the use, necessity and meaning of this definition.

\begin{ejemplo}
The system $x=0$ is a compatible formally evaluative trivializing system. In fact, we have that $x=0$ automatically and then we just proved that its only solution is the trivial one.
\end{ejemplo}

\begin{ejemplo}
The system $x=x$ is not a compatible formally evaluative trivializing system because in the field $\mathbb{Z}_{2}$ it has the solutions $x=1$ and $x=0.$
\end{ejemplo}

\begin{ejemplo}
The system $x=2x$ is a compatible formally evaluative trivializing system because we can subtract $x$ in both sides to get $0=x.$
\end{ejemplo}

\begin{ejemplo}
The system $x=2x^{2}$ is not a compatible formally evaluative trivializing system because in the field $\mathbb{R}$ it has the solutions $x=0$ and $x=\frac{1}{2}.$
\end{ejemplo}

\begin{ejemplo}
Take a field $\mathbb{K}$ with $\car(\mathbb{K})\neq2.$ The system $x=\frac{1}{2}x$ is a compatible formally evaluative trivializing system because we can subtract $\frac{1}{2}x$ in both sides to obtain $x-\frac{1}{2}x=\frac{2}{2}x-\frac{1}{2}x=\frac{2-1}{2}x=\frac{1}{2}x=0,$ which implies $x=0.$
\end{ejemplo}

\begin{remark}
As we saw in the previous example, we are allowed to use the field properties about the multiplication such as distributivity and existence of multiplicative inverses for non-zero elements. What we are not allowed to do is to actually \textit{multiplying} by an arbitrary non-zero element of the field or an unknown. Thus, the following reasoning would be a wrong way of proving that the previous system is a compatible formally evaluative trivializing system:

Take a field $\mathbb{K}$ with $\car(\mathbb{K})\neq2.$ The system $x=\frac{1}{2}x$ is a compatible formally evaluative trivializing system because we can multiply by $2\neq0$ both sides to get $2x=x,$ which we know is compatible formally absolutely trivializing.

The problem here is that we cannot use the multiplication by the scalar $2$ because we are reducing the amount of the field structure that we are able to use, and, precisely, multiplication by an arbitrary scalar is one of these features restricted for us.
\end{remark}

\begin{ejemplo}
The system $x=2x^{2}$ is not a compatible formally evaluative trivializing system because in the field $\mathbb{R}$ it has the solutions $x=0$ and $x=\frac{1}{2}.$
\end{ejemplo}

\begin{ejemplo}\label{QQ}
Consider $\mathbb{K}=\mathbb{Q}.$ The system formed by $x=7y^{2}$ and $y=5x^{2}$ has only the solution $(x,y)=(0,0)$ in $\mathbb{Q}.$ Nevertheless, this system is not compatible formally evaluative trivializing because there is no way to prove that its only solution is the trivial one using only the sum of the field and evaluations of the polynomials $p_{1}=7Y^{2}, p_{2}=5X^{2}$ over the unknowns $x,y$ and the constant $0$ to arrive at evaluations ensuring that the $p_{i}(x,y)$ are forced to be constantly zero by the system. In fact, this is impossible because, if that could be possible to achieve, then this would mean that this system is not having any other solution in any extension of $\mathbb{Q},$ which is not true because it has a non-zero solution in $\mathbb{R}.$
\end{ejemplo}

Thus, in the previous example we see that who is forcing the system to have only the trivial solution is the chosen field and not the structure of the system itself. In the next example we shall see how the structure of the system can be the reason for the triviality of the set of solutions.

\begin{ejemplo}
The system formed by $x = x^{2}-y^{2}$ and $y = x^{2}-y^{2}$ has only the solution $(x,y)=(0,0)$ in whichever field. We can prove that its only solution is the trivial one using only the sum of the field and evaluations of the polynomials $p_{1}=X^{2}-Y^{2}, p_{2}=X^{2}-Y^{2}$ over the unknowns $x,y$ and the constant $0$ to arrive at evaluations ensuring that the $p_{i}(x,y)$ are forced to be constantly zero by the system. We see this. The equations tell us that $x=x^{2}-y^{2}=y$ so $x=p_{1}(x,y)=x^{2}-y^{2}=p_{1}(x,x)=x^{2}-x^{2}=0.$ Thus, $x=0=y$ and we prove it using only the operations allowed. In fact, this tell us that this system is just having the trivial solution whatever field is considered.
\end{ejemplo}

This example is specially important because, as each field has the coefficients involved in this example, it is valid for every field. It is a universal system whose structure determine the triviality of the set of solutions not caring at all about the internal structure of the field. This radical example is easily generalizable to larger dimensions or to more unknowns, equations in the systems or variables.

Coming back to our particular system (\ref{E}), we particularize Definition \ref{DDD}.

\begin{definicion}\label{56}
Let $\mathbb{K}$ be field. A system of equations with coefficients in $\mathbb{K}$ which is not compatible formally evaluative trivializing and whose equations are all in zero-intersect-evolution-axis form
\begin{equation}\label{EE}
\left\{
\begin{aligned}
x_{1} &= a_{11}x_{1}^{2}+\cdots+a_{n1}x_{n}^{2}\\
\vdots\\
x_{n} &= a_{1n}x_{1}^{2}+\cdots+a_{nn}x_{n}^{2}
\end{aligned}
\right.
\end{equation}
is said to be a \textbf{non-trivial zero-intersect-evolution-axis system}.
\end{definicion}

The next important theorem will tell us that, given a positive integer $n,$ if a choice of the structure matrix $(a_{ij})$ with $1\leq i,j\leq n$ in a basis $B$ produces a simple algebra (that is, if this choice verifies the condition S), then this choice generates a non-trivial zero-intersect-evolution-axis system (\ref{EE}). This will be helpful to ensure that our next definition works as we want.

\begin{teorema}\label{THE}
Let $A$ be a simple evolution algebra of dimension $n$ over a field $\mathbb{K}$ with a natural basis $B=\{e_{i}\mid i\in\{1,\dots,n\}\}$ and the multiplication given by \begin{equation}e_{i}e_{i}=\sum_{j=1}^{n}a_{ij}e_{j}\end{equation} for every $i\in\{1,\dots,n\}.$ Then the system of equations
\begin{equation}\label{EEEE}
\left\{
\begin{aligned}
x_{1} &= a_{11}x_{1}^{2}+\cdots+a_{n1}x_{n}^{2}\\
\vdots\\
x_{n} &= a_{1n}x_{1}^{2}+\cdots+a_{nn}x_{n}^{2}
\end{aligned}
\right.
\end{equation}
is a non-trivial zero-intersect-evolution-axis system.
\end{teorema}

\begin{proof}
As $A$ is simple, particularly $e_{1}^{2}\neq0,$ and then $$\id(\{e_{1}^{2}\})=\id(\{\sum_{j=1}^{n}a_{1j}e_{j}\})=A.$$ We will see that $\linspan(\{\sum_{j=1}^{n}a_{ij}e_{j}\mid i\in\{1,\dots,n\}\})=A.$ One inclusion is clear, for the other inclusion we use \cite[Corollary 3.7]{cvs}, which tells us that $$A=\id(\{e_{1}^{2}\})=\linspan(\{e_{i}^{2}\mid i\in D(1)\cup\{1\}\})\subseteq$$\\$$\linspan(\{e_{i}^{2}\mid i\in\{1,\dots,n\}\})=\linspan(\{\sum_{j=1}^{n}a_{ij}e_{j}\mid i\in\{1,\dots,n\}\}).$$ Hence, the matrix $(a_{ij})$ has rank $n.$ In particular, that means that there can be no two equal right sides in the system of equations (\ref{EEEE}) nor a right side equal to zero. Thus, it is not possible to prove that its only solution is the trivial one $(x_{1},\dots,x_{n})=(0,\dots,0)$ using only the sum of the field and evaluations of the polynomials $p_{i}=a_{1i}X_{1}^{2}+\cdots+a_{ni}X_{n}^{2}$ over the unknowns $x_{1},\dots,x_{n}$ and the constant $0$ to arrive at evaluations ensuring that the $p_{i}(x_{1},\dots,x_{n})$ are forced to be constantly zero by the system because we do not have equivalences allowing the first step of the evaluation process permitted. Therefore, we have just proved that (\ref{EEEE}) is not a compatible formally evaluative trivializing system. As, moreover, every equation is in zero-intersect-evolution-axis form, we have that (\ref{EEEE}) is a non-trivial zero-intersect-evolution-axis system, as we wanted to prove.
\end{proof}

If we look at our particular system (\ref{E}), then we know that we should have in mind the conditions S, N and SN when we are thinking about the set of non-zero solutions to this system. The previous theorem allow us to get free from the direct consideration of these problematic conditions when we are dealing with fields presenting ``enough" non-zero solutions for a ``big" subset of the possible choices of coefficients on those systems in a sense that we will make precise in the next definition. This will be extremely helpful for our following discussions and purposes about the study of evolution algebras.

Moreover, when we are studying them, we have to be sure that we are eliminating from our considerations the problematic compatible formally evaluative trivializing systems that we exposed before because, as we showed, they determine its set of non-zero solutions in a way that mostly depends on the structure of the system of equations itself and not much on the structure of the field under consideration. This feature of these systems hinders us from posing questions and making conjectures concerning the behaviour of the set of non-zero solutions for these systems with respect to the fields under consideration. That is, interesting discussions about how are these sets of solutions changing when we make the fields vary along their possible extensions, completions, closures or other constructions could be hidden behind the problematic produced by these pathological systems. Therefore we hope that avoiding these systems we can reach new important insights into the behaviour of the interrelations between the set of non-zero solutions to systems like (\ref{EE}) and the fields under consideration. Thus, to avoid bigger issues, we will focus our attention on systems like (\ref{EE}), that is, on non-trivial zero-intersect-evolution-axis systems.

Now we can think about the set of solutions of non-trivial zero-intersect-evolution-axis systems. We have to be specially careful here because the existence of non-zero solutions for (\ref{EE}) could look easy to check at a first sight when the field $\mathbb{K}$ is algebraically closed using the Hilbert's Nullstellensatz for homogeneous systems of equations as stated in \cite[Theorem 6.1.10]{vvp} but it is easy to realize after looking more in depth that the conditions for the application of this theorem are not fulfilled given the fact that our polynomials are not homogeneous and that we have as many polynomials as variables while the theorem requires one variable more than polynomials. In general, the application of the Hilbert's Nullstellensatz is problematic as a consequence of the existence of the trivial solution and it is also not absolutely satisfactory because it seems that the exigency of the field $\mathbb{K}$ being algebraically closed is too strong here. In fact, one of the questions that we will pose is about what are the conditions necessary or sufficient over the fields so they are acceptable enough to possess non-zero solution to non-trivial zero-intersect-evolution-axis systems.

However, the existence of non-zero solutions for non-trivial zero-intersect-evolution-axis systems with the form given by (\ref{EE}) looks easy to check in the real case (when $\mathbb{K}=\mathbb{R}$) for $n=dim(A)=2.$ The techniques to verify this would involve the use of implicit derivation and/or some geometrical considerations for the nondegenerate conics appearing; the degenerate cases could be studied one by one. Finally, we would obtain that, when (\ref{EE}) is a non-trivial zero-intersect-evolution-axis system, it has always a non-zero solution in the real case (when $\mathbb{K}=\mathbb{R}$) when $\dim(A)=2.$ The techniques that we think that could be used in this case seem to generalize to the complex case (when $\mathbb{K}=\mathbb{C}$) and, also, to any natural dimension $n$ in general because they are related to the nice topological and differentiable classical structures available in $\mathbb{R}^{n}$ and $\mathbb{C}^{n}.$ We will not try to prove this here because it would be too long and far from our main purpose, and also we will pose soon some conjectures that contain the general questions about this topic in the broad sense. In the case of $\mathbb{Q},$ the intuition points to the opposite direction and it seems that (\ref{EE}) is not always solvable (in $\mathbb{Q}$) in a non-trivial way for any given choice of the parameters making (\ref{EE}) a non-trivial zero-intersect-evolution-axis system. Indeed, Example \ref{QQ} is the counterexample proving that this last intuition about $\mathbb{Q}$ is true.

Now we set some terminology to refer to the relation between the fields and the parametric family of systems of equations given by (\ref{EE}) when we choose the parameters in the fields so they are non-trivial zero-intersect-evolution-axis systems.

\begin{definicion}\label{58}Let $\mathbb{K}$ be a field. We say that $\mathbb{K}$ is a \textbf{field under simple evolution algebras of dimension $n$ with non-zero idempotents} or an \textbf{$n$-FSEANI} if the parametric family of systems of equations given by (\ref{EE}) has a non-zero solution in the field $\mathbb{K}$ for every choice of the parameters $a_{ij}$ in the field $\mathbb{K}$ making (\ref{EE}) a non-trivial zero-intersect-evolution-axis system.\end{definicion}

This definition, together with the Theorem \ref{THE}, tells us that a field $\mathbb{K}$ which is an $n$-FSEANI can be used as the underlying field of an evolution algebra $A$ of dimension $n$ with non-zero idempotents for any choice of the structure matrix (for a given natural basis $B$) subjected to the condition of making the algebra $A$ simple and (\ref{EE}) a non-trivial zero-intersect-evolution-axis system. Following this definition we have that Example \ref{Exa} proves that $\mathbb{Z}_{3}$ is not a $2$-FSEANI. Another easy example is given by $\mathbb{Q},$ which is clearly not a $2$-FSEANI, as the system given by the equations $x=3y^2$ and $y=2x^2$ has no non-zero solution in the field $\mathbb{Q},$ although this equation does have a non-zero solution in an extension of $\mathbb{Q}.$

\begin{definicion}\label{59}
We say that a choice of parameters in (\ref{EE}) or its associated coefficient matrix $M:=(a_{ij})_{i,j\in\{1,\dots,n\}}$ \textbf{gives a non-zero idempotent in its associated evolution algebra over $\mathbb{K}$} or that $M$ is a \textbf{$\mathbb{K}$-nnidempassevalg matrix (choice)} if the parametric family of systems of equations given by (\ref{EE}) has a non-zero solution over $\mathbb{K}$ for this matrix (choice) of parameters.
\end{definicion}

With these notations and terminology established, we are able to set our conjectures.

\begin{conjetura}\label{Con}
The fields $\mathbb{R}$ and $\mathbb{C}$ are $n$-FSEANIs for every $n\in\mathbb{N}.$
\end{conjetura}

We could ask more questions about what type of fields verify this condition for a given dimension $n$ or for special choices. It seems (but it is not clear) that algebraic closedness is not necessary but it is sufficient. Also, a real closure of the ordered field of rational numbers, that is, for example, the field $\mathbb{R}_{\alg}$ of real algebraic numbers, is a strong candidate to be an $n$-FSEANIs for every $n\in\mathbb{N}.$ Thus, this tells us that maybe it is enough with other kind of closures and leads our questions to the direction of considering different closures for the fields. We could even try to extend our considerations from idempotens ($e^{2}=e$) to other special elements verifying other interesting identities. This topic is interesting and deserves a deeper study and analysis than the one we are able to do here because this falls out of the scope of this work.

Until this point, this discussion has allowed us to see that it is not true that, for every field $\mathbb{K},$ every simple evolution algebra over $\mathbb{K}$ is generated by an idempotent, as showed by our Example \ref{Exa}. Nonetheless, we suspect that every simple evolution algebra over $\mathbb{K}$ is generated by an idempotent in very important an representative cases, such as simple evolution algebras over $\mathbb{K}=\mathbb{R}$ or over $\mathbb{K}=\mathbb{C},$ which is what we wrote in the Conjecture \ref{Con}. With this in mind, we can pass to the next topic.

The restriction to simple algebras was justified by the existence of counterexamples among them to the direct equivalent result in the frame of evolution algebras to the fact expressed in Theorem \ref{F} and widely known in associative ring theory. However, this restriction is very strong for our next purposes and we can use what we learned in the previous pages of this section to go beyond this restriction under adequate hypothesis.

\begin{definicion}\label{61}
Let $A$ be an algebra and $\mathcal{I}$ the set of its minimal ideals among those which are different from $\{0\}.$ Then the \textbf{(ideal) socle of $A$} is defined as the sum of the minimal ideals of $A$ in $\mathcal{I},$ that is, $\soc(A):=\sum_{I\in\mathcal{I}}I.$ Now let $\mathcal{E}$ be the set of the minimal evolution ideals of $A$ among those evolution ideals of $A$ which are different from $\{0\}.$ Then the \textbf{(ideal) evolution socle of $A$} is defined as the sum of the minimal evolution ideals of $A$ in $\mathcal{E},$ that is, $\ev \soc(A):=\sum_{E\in\mathcal{E}}E.$ Where, when the sums are empty, we adopt the usual convention that says that an empty sum is equal to the trivial ideal $\{0\}.$
\end{definicion}

It is clear that, if $A$ is an algebra of finite dimension, every minimal evolution ideal $E$ of $A$ among those evolution ideals of $A$ which are different from $\{0\}$ verifies that either $E$ is itself a minimal ideal among all the ideals of $A$ different from $\{0\}$ or all the non-zero ideals $I$ of $A$ contained in $E$ are not evolution ideals (i.e., these ideals do not have natural bases). Thus, for every minimal evolution ideal $E$ of $A$ among those evolution ideals of $A$ which are different from $\{0\}$ we can consider its \textbf{$A$-socle}, denoted $A\mbox{-soc}(E),$ as the sum of the minimal ideals of $A$ among all the ideals of $A$ different from $\{0\}$ and contained in $E.$ Clearly, $A\mbox{-soc}(E)\subseteq\soc(A).$

\begin{warning}\label{62}
These definitions could lead us to the confusion of thinking that we have $\soc(A)\subseteq\ev\soc(A)$ because every minimal ideal $I$ of $A$ is contained in some minimal evolution ideal $E$ among the evolution ideals of $A,$ where $E$ can be chosen from $\MinEvid(I).$ But this last sentence is not evident and it could not be true. Indeed, it could happen that the minimal ideal $I$ of $A$ is contained in an evolution ideal $E$ of $A$ from $\MinEvid(I)$ such that $E,$ although minimal among the evolution ideals containing $I,$ has inside it another non-zero evolution ideal $G\subseteq E$ such that $I\nsubseteq G,$ which would imply that $E$ is not a minimal evolution ideal $E$ of $A$ among those evolution ideals of $A$ which are different from $\{0\}$ because $G$ would be a non-zero evolution ideal smaller than $E.$
\end{warning}

This leads us also to another interesting concept. Given a minimal ideal $I$ of $A$ we can consider its $A\mbox{-ev}\soc(I)$ as the sum of all the minimal evolution ideals among those evolution ideals of $A$ which are different from $\{0\}$ that can be found inside the ideals of $\MinEvid(I).$

As one can note promptly, the $A$-evolution socle works as an \textit{evolution dual} notion of the $A$-socle. Surprisingly, the different order structures emerging from the sets of substructures of evolution algebras present these kind of aspects that we think are not only a feature of the evolution attribute but something \textit{deeper} concerning the order inside these sets.

The relations between the different socles defined above has to be studied more in depth. In fact, we let here an important question for the future concerning the warning above.

\begin{cuestion}\label{63}
What are the conditions necessary over the algebra $A$ in order to have $\soc(A)\subseteq\ev\soc(A).$
\end{cuestion}

Unfortunately, as a consequence of a lack of space and time at this moment, we have to pass to the next topic. However, at the end of this section we will be able to say more about this question, although we will not be able to solve it completely.

Now we have two possibilities depending on how strong we want to set our hypothesis. The first possibility is to follow exactly what the examples show to us and restrict ourselves to the study of evolution ideals (this is because in our previous results we have used simple evolution algebras, which, of course, are an example of minimal ideals of themselves having natural bases). This would lead us to the two following definitions

\begin{definicion}Let $A$ be an algebra over a field $\mathbb{K}$ verifying that every minimal evolution ideal $I$ (i.e., minimal among the evolution ideals) of $A$ is generated by an idempotent. We say that $A$ is an \textbf{algebra with minimal evolution ideals generated by idempotents}.\end{definicion}

\begin{definicion}Let $A$ be an algebra over a field $\mathbb{K}$ verifying that every evolution ideal $I$ (i.e., with a natural basis) of $A$ is generated by an idempotent. We say that $A$ is an \textbf{algebra with evolution ideals generated by idempotents}.\end{definicion}

The study of these types of structures would lead us towards the neighbourhoods of the evolution socle but this would give us not a lot of information about its socle, at least at a first sight. For this reason we will choose a way with stronger hypothesis. This way avoids the restriction on evolution ideals and leads us to set the following concepts.

\begin{definicion}
Let $A$ be an algebra over a field $\mathbb{K}$ verifying that every minimal ideal $I$ with a natural basis (i.e., which is also an evolution ideal) of $A$ is generated by an idempotent. We say that $A$ is an \textbf{algebra with minimal and evolution ideals generated by idempotents}.
\end{definicion}

\begin{definicion}
Let $A$ be an algebra over a field $\mathbb{K}$ verifying that every minimal ideal $I$ of $A$ is generated by an idempotent. We say that $A$ is an \textbf{algebra with minimal ideals generated by idempotents}.
\end{definicion}

Also, it is clear that if $A$ is an algebra with minimal ideals generated by idempotents, then $A$ is an algebra with minimal and evolution ideals generated by idempotents. In the opposite direction it is clear that, if $A$ is an algebra with minimal and evolution ideals generated by idempotents whose minimal ideals are all evolution ideals (i.e., have a natural basis), then $A$ is an algebra with minimal ideals generated by idempotents.

Now, in order to show the high number and variety of evolution algebras available, given $n\in\mathbb{N},$ we will give some example of evolution algebras having minimal ideals with natural bases of every dimension between $1$ and $n$ (first and second example below), and, at least, with every odd dimension between $1$ and $n$ (third example below). With this we want to show that the restrictions made in the definitions before are necessary to have the possibility to understand the big zoo of this kind of algebras.

\begin{ejemplo}
we fix $n\in\mathbb{N},$ a field $\mathbb{K}$ and consider $\mathbb{K}^{\frac{(n+1)n}{2}}.$ Then it is possible to give to $\mathbb{K}^{\frac{(n+1)n}{2}}$ an evolution algebra structure over $\mathbb{K}$ having minimal evolution ideals of every dimension less than or equal to $n.$ Take a basis $B$ of $\mathbb{K}^{\frac{(n+1)n}{2}}$ and call the elements $B:=\{e_{ij}\mid 1\leq i\leq n, 1\leq j\leq i\}$ so $\Card(B)=\sum_{i=1}^{n}i=\frac{(n+1)n}{2}.$ Now, for every $1\leq i\leq n$ and every $1\leq j\leq i,$ set $e_{ij}e_{ij}=\sum_{k=1}^{i}a_{ijk}e_{ik}$ while $e_{ij}e_{i'j'}=0$ if $(i,j)\neq(i',j').$ Then it is clear that, for every $1\leq i\leq n,$ the set $\linspan(\{e_{ij}\mid 1\leq j\leq i\})$ is an evolution ideal of dimension $i.$

In particular, if we choose, for every $1\leq i\leq n$ and every $1\leq j\leq i,$ the product given by $e_{ij}e_{ij}=e_{i,j+1}$ if $j<i$ and $e_{ii}e_{ii}=e_{i1},$ then we obtain an evolution algebra having every set $\linspan(\{e_{ij}\mid 1\leq j\leq i\})$ as a minimal evolution ideal of dimension $i.$

Suppose $\chara(\mathbb{K})>n+2.$ In particular, if we choose, for every $1\leq i\leq n$ and every $1\leq j\leq i,$ the product given by $e_{ij}e_{ij}=je_{ij}+(j+1)e_{i,j+1}$ if $j<i$ and $e_{ii}e_{ii}=ie_{ii}+(i+1)e_{i1},$ then we obtain an evolution algebra having every set $\linspan(\{e_{ij}\mid 1\leq j\leq i\})$ as a minimal evolution ideal of dimension $i$ and non-zero elements in the diagonal of the structure matrix. We prove now this last statement. If $i=1,$ it is immediate. We fix $1< i\leq n$ arbitrary. To see that $\linspan(\{e_{ij}\mid 1\leq j\leq i\})$ is a minimal evolution ideal it is enough to prove that every non-zero element $w\in\linspan(\{e_{ij}\mid 1\leq j\leq i\})$ generate $\linspan(\{e_{ij}\mid 1\leq j\leq i\})$ as an ideal of $A$ because it is clear by construction that $\linspan(\{e_{ij}\mid 1\leq j\leq i\})$ is an evolution ideal, so we just need to prove the minimality. The construction of the ideals is designed to allow a telescopic strategy for the proof on this fact. We fix a non-zero element $w\in\linspan(\{e_{ij}\mid 1\leq j\leq i\})$ arbitrary. Then there exist $1\leq j\leq i$ such that $\proj{B,e_{ij}}(w)\neq 0.$ Therefore, reordering if necessary and multiplying by a scalar, we can suppose, without lost of generality, that $\proj{B,e_{i1}}(w)=1.$ Then $we_{i1}=e_{i1}^{2}=e_{i1}e_{i1}=1e_{i1}+(1+1)e_{i,1+1}\in\id(\{w\}).$ Inductively, and multiplying by scalars when necessary, we obtain $e_{ij}e_{ij}=je_{ij}+(j+1)e_{i,j+1}, e_{ii}e_{ii}=ie_{ii}+(i+1)e_{i1}\in\id(\{w\})$ for all $1\leq j< i.$ Now the sum $\sum_{j=1}^{i}(-1)^{j-1}(e_{ij}e_{ij})=\sum_{j=1}^{i-1}(-1)^{j-1}(je_{ij}+(j+1)e_{i,j+1})+(-1)^{i-1}(ie_{ii}+(i+1)e_{i1})=e_{i1}+(-1)^{i-1}(i+1)e_{i1}=le_{i1},$ for some scalar $l\neq 0$ since $1+(-1)^{i-1}(i+1)\neq 0$ as $i>1$ and $\chara(\mathbb{K})>n+2,$ which implies automatically that $\id(\{w\})=\linspan(\{e_{ij}\mid 1\leq j\leq i\}).$

Suppose $\chara(\mathbb{K})>2.$ In particular, if we choose, for every $1\leq i\leq n$ and every $1\leq j\leq i,$ the product given by $e_{ij}e_{ij}=e_{ij}+e_{i,j+1}$ if $j<i$ and $e_{ii}e_{ii}=e_{ii}+e_{i1},$ then, for every $i$ odd, we obtain an evolution algebra having every set $\linspan(\{e_{ij}\mid 1\leq j\leq i\})$ as a minimal evolution ideal of dimension $i$ and a structure matrix with diagonal constantly $1.$ The strategy to see this is similar to the one followed in the previous example.
\end{ejemplo}

There are several ways to follow at this point. We decide to study what conditions are sufficient to guarantee that a nondegenerate evolution algebra $A$ which is also an algebra with minimal and evolution ideals generated by idempotents is an algebra with (all its) minimal ideals generated by idempotents.

\begin{definicion}\label{69}
Let $A$ be an algebra over a field $\mathbb{K}.$ Consider an element $a\in A$ and scalars $f,g\in\mathbb{K}$ such that there exist a set of $m$ different non-zero idempotents $E=\{e_{1},\dots,e_{m}\}\subseteq A$ verifying that $a=\sum_{i=1}^{m}ge_{i}$ and that, for every $k\in\{1,\dots,m\},$ there exist exactly one pair $(i,j)$ with $i\neq j$ such that $e_{i}e_{j}=fe_{k},$ while the rest of the choices of pairs $(i,j)$ with $i\neq j$ are zero. Then we say that $a$ is a \textbf{$g;f$-idemelement}.
\end{definicion}

The next proposition is immediate.

\begin{proposicion}\label{70}
Let $A$ be an algebra over a field $\mathbb{K}.$ Consider an element $a\in A$ and non-zero scalars $f,g\in\mathbb{K}$ such that $a$ is a $g;f$-idemelement and there exist a choice for $\sqrt{\frac{g}{g^{2}f+g^{2}}}$ within the field $\mathbb{K}.$ Then $\sqrt{\frac{g}{g^{2}f+g^{2}}}a$ is an idempotent for every possible choice of $\sqrt{\frac{g}{g^{2}f+g^{2}}}.$
\end{proposicion}

\begin{proof}
We have that $\sqrt{\frac{g}{g^{2}f+g^{2}}}a\sqrt{\frac{g}{g^{2}f+g^{2}}}a=\frac{g}{g^{2}f+g^{2}}a^{2}=\frac{g}{g^{2}f+g^{2}}(\sum_{i=1}^{m}ge_{i})^{2}$ for some set of $m$ different non-zero idempotents $E=\{e_{1},\dots,e_{m}\}\subseteq A$ verifying that $a=\sum_{i=1}^{m}ge_{i}$ and that, for every $k\in\{1,\dots, m\},$ there exist exactly one pair $(i,j)$ with $i\neq j$ such that $e_{i}e_{j}=fe_{k},$ while the rest of the choices of pairs $(i,j)$ with $i\neq j$ are zero. Then $(\sum_{i=1}^{m}ge_{i})^{2}=\sum_{i,j=1}^{m}g^{2}e_{i}e_{j}=\sum_{k=1}^{m}g^{2}fe_{k}+\sum_{i=1}^{m}g^{2}e_{i}=\sum_{k=1}^{m}(g^{2}f+g^{2})e_{k}.$ Thus, $\frac{g}{g^{2}f+g^{2}}(\sum_{i=1}^{m}ge_{i})^{2}=\frac{g}{g^{2}f+g^{2}}\sum_{k=1}^{m}(g^{2}f+g^{2})e_{k}=\sum_{i=1}^{m}ge_{i}=a.$
\end{proof}

It is obvious that it is possible to extend the definition of $g;f$-idemelements to other kinds of configurations and elements distinguished by identities to study the prevalence and relations between idempotents and other elements distinguished by interesting identities but this study is beyond our scope in this work.

The next proposition follows automatically.

\begin{proposicion}\label{71}
Let $I$ be a minimal ideal of an algebra $A$ over a field $\mathbb{K}$ verifying that there exist a non-zero $g;f$-idemelement $0\neq a\in I$ for some $f,g\in\mathbb{K}$ such that there exist a choice for $\sqrt{\frac{g}{g^{2}f+g^{2}}}.$ Then $\sqrt{\frac{g}{g^{2}f+g^{2}}}a$ is a non-zero idempotent generating $I$ for every possible choice of $\sqrt{\frac{g}{g^{2}f+g^{2}}}.$
\end{proposicion}

The previous results are useful in very general situations. The next definition is also highly general.

\begin{definicion}\label{72}
Let $A$ be an algebra, $\mathcal{E}\subseteq\mathcal{P}(A)\smallsetminus\{\emptyset\}$ a set of subsets of $A$ and $\emptyset\neq U\subseteq A$ another non-empty subset of $A.$ The \textbf{left annihilator of $\mathcal{E}$ in the set $U$} is the set $\lann_{U}(\mathcal{E}):=\{u\in U\mid uF=\{0\} \mbox{\ for all\ } F\in\mathcal{E}\}.$ The \textbf{left autoannihilator for products of $\mathcal{E}$} is the set $\autolann(\mathcal{E})=\bigcup_{E\in\mathcal{E}}\{e\in E\mid eF=\{0\} \mbox{\ for all\ } F\in\mathcal{E}\mbox{\ with \ } E\neq F\}.$ When the algebra is commutative, we can simplify the notation to $\ann_{U}(\mathcal{E})$ and $\autoann(\mathcal{E}),$ respectively. Let $E$ be a subset of the algebra $A.$ We denote the set of non-zero idempotents of $A$ in $E$ by $\idem(E):=\{e\in E\mid e^{2}=e\neq 0\}$ and the set of non-zero minimal idempotents of $A$ in $E$ by $\minidem(E).$
\end{definicion}

It is clear that $\autolann(\mathcal{E})=\cup_{U\in\mathcal{E}}\lann_{U}(\mathcal{E}\smallsetminus\{U\}).$ The next definition will be useful to formulate the theorem after it. Before introducing this definition, we set some notation.

\begin{notacion}
Let $(A,+)$ be an associative magma (i.e., a semigroup), $n\in\mathbb{N}$ a natural number, $E\subseteq A$ a subset and $\mathcal{E}\subseteq\mathcal{P}(A)\smallsetminus\{\emptyset\}.$ We denote $\sumspan(E,n):=\{\sum_{i=1}^{n}e_{i}\mid e_{i}\in E\}$ and \begin{gather*}
\sumspan(\mathcal{E},E,n):=\left\{ \sum_{i=1}^{n}e_{i} \ \middle\vert \begin{array}{l}
   e_{i}\in E \mbox{ and, for every $F\in\mathcal{E}$, there exists}\\
    \mbox{at most one $i\in\{1,\dots n\}$ with } e_{i}\in F\end{array}\right\}.
\end{gather*} We also denote $\sumspan(\mathcal{E},E):=\cup_{n\in\mathbb{N}}\sumspan(\mathcal{E},E,n)$ and $\sumspan(E):=\cup_{n\in\mathbb{N}}\sumspan(E,n).$
\end{notacion}

As it is easy to see, $\sumspan(E)$ is just $\linspan(E)$ restricting the multiplication to the scalar $1\in\mathbb{K}$ when we are in the context of a vector space.

\begin{definicion}\label{74}
Let $I, B$ be two subsets of an algebra $A$ over a field $\mathbb{K}$ such that there exists $\mathcal{E}\subseteq\mathcal{P}(B)\smallsetminus\{\emptyset\}$ verifying $$\emptyset\neq T:= (I\cap\sumspan(\autolann(\mathcal{E})))\smallsetminus\{0\}.$$ Then we say that \textbf{$I$ admits a non-zero sum combination of $B$ following the nullity in $\mathcal{E}$}. The set $T$ is called the \textbf{set in $I$ of non-zero sum combination of $B$ following the nullity in $\mathcal{E}$} and we denote it as $T:=\snzsc(I, B,\mathcal{E}).$
\end{definicion}

Now the theorem follows immediately.

\begin{teorema}\label{75}
Let $I$ be a minimal ideal of an algebra $A$ over a field $\mathbb{K}.$ Suppose that there exists $\mathcal{E}\subseteq\mathcal{P}(\idem(A))\smallsetminus\{\emptyset\}$ such that $$\emptyset\neq T:= (I\cap\sumspan(\mathcal{E},\autolann(\mathcal{E})))\smallsetminus\{0\}.$$ That is, $I$ is a minimal ideal which admits a non-zero sum combination of $\idem(A)$ following the nullity in $\mathcal{E}$. Then the elements in the subset of $I$ given by $T=\snzsc(I,\idem(A),\mathcal{E})$ are idempotents generating $I.$
\end{teorema}

\begin{proof}
Let $e$ be an element in $T.$ Then $0\neq e=\sum_{i=1}^{n}e_{i}$ for some $e_{i}\in \autolann(\mathcal{E})$ and $n\in\mathbb{N}.$ Also, for every $F\in\mathcal{E}$ there exists, at most, one $i\in\{1,\dots n\}$ with $e_{i}\in F.$ Therefore, $e^{2}=(\sum_{i=1}^{n}e_{i})^{2}=\sum_{i=1}^{n}e_{i}=e$ since $e_{i}e_{j}=0$ for $i\neq j,$ as they are necessarily in different sets of $\mathcal{E},$ and $e_{i}e_{i}=e_{i},$ as $e_{i}\in\idem(A).$ As $e$ is an non-zero idempotent in $I,$ which is minimal, the theorem is proved.
\end{proof}

Another way to obtain nice properties about evolution ideals of algebras is strengthening the simplicity of these ideals.

\begin{definicion}\label{76}
Let $I$ be an ideal of an algebra $A$ over a field $\mathbb{K}.$ We say that $I$ is a \textbf{simple ideal} if it is simple considered as an algebra. An algebra $A$ such that all its minimal ideals are simple is called an \textbf{minimal-ideally simple algebra}. An algebra $A$ such that all its minimal ideals are simple and have natural bases is called an \textbf{evolution-minimal-ideally simple algebra}. An algebra $A$ such that all its minimal ideals are simple with non-zero product and have natural bases is called an \textbf{evolution-minimal-ideally simple algebra with non-zero product in its minimal ideals}.
\end{definicion}

We note that a minimal ideal with the product constantly zero has dimension $1$ necessarily.

\begin{ejemplo}
Examples of these kind of algebras can be easily obtained as direct sums of simple algebras. So, given two simple evolution algebras with non-zero product $A_{1}, A_{2},$ the direct sum algebra $A_{1}\oplus A_{2}$ with the product given by $(x_{1},x_{2})(y_{1},y_{2})=(x_{1}y_{1},x_{2}y_{2})$ verifies that all its minimal ideals ($A_{1},A_{2}$) are simple with non-zero product and have natural bases.
\end{ejemplo}

\begin{remark}
It is obvious that, if a minimal ideal $I$ of an (evolution) algebra $A$ is also an extension evolution ideal of $A$ (cf. Definition \ref{EEI}), then $I$ is a simple ideal of $A$ because the products of elements in $I$ multiplied by elements in $A$ (i.e., $AI$) reduce, actually, to products which, in fact, are internal to $I$ as a consequence of the orthogonality given by the natural basis in $A$ extending the basis of $I.$ Indeed, it is not necessary for $I$ to be an extension evolution ideal to make this same reasoning work, it is sufficient for $I$ to be a minimal ideal of $A$ having an orthogonal complement in $A.$
\end{remark}

A question arising now is if every minimal (evolution) ideal of an (evolution) algebra is simple. The answer is no as shows the next example.

\begin{ejemplo}
Let $\mathbb{K}$ be a field and consider the evolution algebra $A$ over $\mathbb{K}$ with basis $B=\{e_{1},e_{2},e_{3},e_{4}\}$ and product given by $e_{1}^{2}=-(e_{3}+e_{4}),e_{2}^{2}=-(e_{1}+e_{2}),e_{3}^{2}=e_{1}+e_{2}, e_{4}^{2}=e_{3}+e_{4}.$ We affirm that $I:=\linspan(\{e_{1}+e_{2},e_{3}+e_{4}\})=\{ae_{1}+ae_{2}+be_{3}+be_{4}\mid a,b\in\mathbb{K}\}$ is a minimal ideal of $A.$ We prove it. The vector subspace $I$ is an ideal of $A$ because $e_{1}(e_{1}+e_{2})=-(e_{3}+e_{4}),e_{2}(e_{1}+e_{2})=-(e_{1}+e_{2}),e_{3}(e_{1}+e_{2})=0,e_{4}(e_{1}+e_{2})=0,e_{1}(e_{3}+e_{4})=0,e_{2}(e_{3}+e_{4})=0,e_{3}(e_{3}+e_{4})=e_{1}+e_{2},e_{4}(e_{3}+e_{4})=e_{3}+e_{4}.$ Also, $I$ is minimal. To see this, we fix a non-zero element $0\neq w= ae_{1}+ae_{2}+be_{3}+be_{4}\in I$ chosen arbitrarily. There are two options. On one hand, suppose $a\neq 0.$ Then $e_{1}(ae_{1}+ae_{2}+be_{3}+be_{4})=-a(e_{3}+e_{4})$ and $e_{2}(ae_{1}+ae_{2}+be_{3}+be_{4})=-a(e_{1}+e_{2}),$ which implies that the element $w$ generates $I.$ On the other hand, suppose $b\neq 0.$ Then $e_{4}(ae_{1}+ae_{2}+be_{3}+be_{4})=b(e_{3}+e_{4})$ and $e_{3}(ae_{1}+ae_{2}+be_{3}+be_{4})=b(e_{1}+e_{2}),$ which implies that the element $w$ generates $I.$ This proves that $I$ is a minimal ideal of the algebra $A.$ In fact, as $B'=\{e_{1}+e_{2},e_{3}+e_{4}\}$ works as a natural basis for $I,$ then $I$ is also an evolution ideal. Nevertheless, $I$ considered as an algebra is not simple. We see this. Consider the linear subspace of $I$ generate by the element $e_{1}+e_{2}+e_{3}+e_{4}\in I$ which corresponds to take $a=b$ in the previous description of $I,$ i.e., $\linspan(e_{1}+e_{2}+e_{3}+e_{4})=\{ae_{1}+ae_{2}+ae_{3}+ae_{4}\mid a\in\mathbb{K}\}:=J\subseteq I.$ We affirm that $J$ is an (evolution) ideal of $I.$ This is true because $(e_{1}+e_{2})(ae_{1}+ae_{2}+ae_{3}+ae_{4})=a(e_{1}^{2}+e_{2}^{2})=-(ae_{1}+ae_{2}+ae_{3}+ae_{4}), (e_{3}+e_{4})(ae_{1}+ae_{2}+ae_{3}+ae_{4})=a(e_{3}^{2}+e_{4}^{2})=ae_{1}+ae_{2}+ae_{3}+ae_{4}\in J.$ This proves that $I$ is not simple considered as an algebra because $J\neq\{0\}$ is an (evolution) ideal of $I$ which is different from $I$ itself.
\end{ejemplo}

With our previous study, we get the following theorem.

\begin{teorema}\label{80}
Let $A$ be an evolution-minimal-ideally simple algebra with non-zero product in its minimal ideals with $\dim(A)=n$ over a field $\mathbb{K}$ which is an $m$-FSEANI for every positive natural $m\leq n.$ Then every minimal ideal of the algebra $A$ is generated by an idempotent.
\end{teorema}

\begin{proof}
Let $I$ be a minimal ideal of the algebra $A.$ As the algebra $A$ is a minimal-ideally simple evolution algebra, the ideal $I$ is simple considered as an algebra. Therefore, the ideal $I$ is a simple algebra of dimension $m\leq n$ over $\mathbb{K},$ which is an $m$-FSEANI. Thus, as the ideal $I$ has non-zero product and $\mathbb{K}$ is an $m$-FSEANI, the ideal $I$ has a non-zero idempotent, which is necessarily a generator given the minimality of the ideal $I$ as an ideal of the algebra $A.$ Moreover, the idempotent obtained would generate the ideal $I$ inside itself (without the intervention of elements of the algebra $A$ not contained in the ideal $I$) given the fact that the ideal $I$ is simple viewed as an algebra.
\end{proof}

\begin{definicion}\label{81}Let $A$ be an algebra and define $\MinId(A)$ the set of all the minimal ideals of $A.$ A \textbf{coherent choice of minimal elements in the algebra $A$} is a function $f\colon\MinId(A)\to A$ such that $f(I)$ generates the minimal ideal $I.$ We call $\linspan(f(\MinId(A)))$ the \textbf{linear space associated to the coherent choice of minimal elements $f$}.\end{definicion}

\begin{proposicion}\label{82}
Let $A$ be an algebra and $f$ a coherent choice of minimal elements in the algebra $A.$ Then $\soc(A)=\sum_{I\in\MinId(A)}I=\sum_{I\in\MinId(A)}\id(\{f(I)\}).$ If the linear space associated to the coherent choice $f$ has dimension $m$ and $n=\dim(\soc(A)),$ then there exist at most $n-m$ minimal ideals $I$ whose elements $e:=f(I)$ associated via the coherent choice $f$ verify that there exist $a\in A$ such that $ae$ or $ea$ are not in the vector space $\linspan(\{e\}).$
\end{proposicion}

\begin{proof}
Let $f$ be a coherent choice of minimal elements in $A$ and suppose there exists $n-m+1$ minimal ideals $I$ whose elements $e:=f(I)$ associated via the coherent choice $f$ verify that there exist $a\in A$ such that $ae$ or $ea$ are not in $\linspan(\{e\}).$ Then there exist $n-m+1$ minimal ideals $I$ having dimension greater than $1.$ For each one of these minimal ideals $I,$ take an element which is not in the linear subspace generated by $f(I).$ Moreover, as we know, given two different minimal ideals $I,J$ we have that $I\cap J=\{0\}$ as a consequence of the minimality. Thus, with this construction, we could find $n-m+1+m=n+1$ linearly independent elements in $\soc(A),$ which contradicts $\dim(\soc(A))=n.$
\end{proof}

\begin{definicion}\label{83}
Given a basis $B$ of an algebra $A$ and two elements $c,d\in A,$ we say that $B$ \textbf{separates} $c$ and $d$ if $B=C\cup D$ with $c\in\linspan(C)$ and $d\in\linspan(D).$
\end{definicion}

\begin{definicion}\cite[Definitions 3.2]{bcm}\label{84}
Let $A$ be an evolution algebra and $a\in A.$ We say that $a$ is a \textbf{natural element} of the evolution algebra $A$ if there exist a natural basis $B$ of $A$ such that $a\in B.$
\end{definicion}

We note that the previous definition implies automatically that $0$ cannot be a natural element. The result that appears in the next proposition is immediate and we will use part of it in the discussion following it.

\begin{proposicion}\label{85}
Let $A$ be an evolution algebra over a field $\mathbb{K}.$ Every natural idempotent of $A$ is a minimal idempotent. In particular, if $e\in A$ is a natural idempotent, then the ideal $Ae=eA=e(Ae)=(eA)e=\mathbb{K}e=\id(\{e\})$ is minimal. Also, if $0\neq a\in A$ verifies $Aa=\mathbb{K}a,$ then $a^{2}=aa=ra,$ for one $r\in\mathbb{K},$ and $\frac{1}{r}a$ is a minimal idempotent (not necessarily natural) when $r\neq 0.$
\end{proposicion}

\begin{proof}
Let $e\in A$ be a natural idempotent of $A.$ Then there exists a natural basis $B=\{e_{i}\mid i\in\Lambda\}$ generating $A$ via linear combinations such that $e:=e_{j}\in B$ for some $j\in\Lambda.$ In particular, $A$ is an evolution algebra. Then every element $a\in A$ can be written as $a=\sum_{i\in\Lambda}a_{i}e_{i}$ with $a_{i}\neq0$ just for a finite number of elements in $\Lambda.$ Then $ae=ea=e(ae)=(ea)e=a_{j}e$ so $\id(\{e\})=Ae=eA=e(Ae)=(eA)e=\mathbb{K}e$ has dimension $1$ and $e$ is a minimal idempotent.

On the other hand, if $Aa=\mathbb{K}a,$ then $\id(\{\frac{1}{r}a\})=\id(\{a\})=\mathbb{K}a$ has dimension $1$ so $\frac{1}{r}a$ is a minimal element verifying $\frac{1}{r}a\frac{1}{r}a=\frac{1}{r^{2}}aa=\frac{1}{r^{2}}ra=\frac{1}{r}a.$ Thus, $\frac{1}{r}a$ is a minimal idempotent when $r\neq 0.$
\end{proof}

Given two minimal idempotents $e$ and $u$ in an evolution algebra $A,$ with $\dim(A)=n,$ generating the ideals $E$ and $U$ with $E\neq U,$ respectively, such that $e$ is natural and $u$ is just an idempotent, we are interested in their relation. First of all, $eu=ue\in I\cap E=\{0\}$ so $eu=0.$ This implies $e+u$ is also a non-zero idempotent. As $e$ is natural, it belongs to a natural basis $B:=\{e_{1},\dots,e_{n}\}$ with $e_{1}=e.$ If we write $u=\sum_{j=1}^{n}a_{j}e_{j},$ then the first thing that we know for sure is that $a_{1}=0$ because $0=eu=e\sum_{j=1}^{n}a_{j}e_{j}=e_{1}\sum_{j=1}^{n}a_{j}e_{j}=a_{1}e_{1}.$ That means that $u\in\linspan (B\smallsetminus\{e_{1}\}).$ That is, we have just proved that there exists always a natural basis separating a natural idempotent from any other minimal idempotent and these \textit{separating} bases (there could be more than one) are precisely those ones that makes $e$ a natural element. Thus, given a set of $m$ different natural (and, as we have proved, minimal) idempotents $\{e^{k}\mid 1\leq k\leq m\},$ we can find $m$ natural bases $B^{k}$ with $e^{k}\in B^{k}$ of $A,$ for $1\leq k\leq m,$ such that $u\in\cap_{k=1}^{m}\linspan(B^{k}\smallsetminus\{e^{k}\}).$ This provide us with knowledge about the \textit{confinement} of idempotents by natural bases of $A$ in the evolution algebra $A$ given the existence of other natural idempotents.

We remember that the minimal ideal of the evolution algebra $A$ generated by a natural idempotent of $A$ has necessarily dimension $1$ as we proved in the previous proposition. Now, it is easy to see that, if we choose $m$ different natural idempotents $e^{k},$ for $1\leq k\leq m,$ generating different minimal ideals of the evolution algebra $A,$ then we can collect them in a natural basis $\{e^{k}\mid 1\leq k\leq m\}$ generating by linear combination the evolution ideal $\alg(\{e^{k}\mid 1\leq k\leq m\})=\id(\{e^{k}\mid 1\leq k\leq m\})=\linspan(\{e^{k}\mid 1\leq k\leq m\})$ of the evolution algebra $A.$ This tells us that, given $m$ different minimal ideals generated by a natural idempotent in an evolution algebra $A,$ it is always possible to construct an evolution ideal $I$ of $A$ of dimension $m$ having the set formed by the chosen natural idempotents generating these $m$ different minimal ideals of the evolution algebra $A$ as its natural basis. One interesting questions to analyze in the future is when $I=A.$

However and parallelly, if we just choose $m$ different minimal (not necessarily natural) idempotents  $e^{k},$ for $1\leq k\leq m,$ generating different minimal ideals (which now could have dimension greater than $1$) of the evolution algebra $A,$ then we can collect them all in a natural basis $\{e^{k}\mid 1\leq k\leq m\}.$ But now this basis is not necessarily a basis of the ideal generated by these elements $\id(\{e^{k}\mid 1\leq k\leq m\})$ because it could happen that $\id(\{e^{k}\mid 1\leq k\leq m\})\neq\linspan(\{e^{k}\mid 1\leq k\leq m\})$ as a consequence of the possibility of the existence of an element $e^{i}\in\{e^{k}\mid 1\leq k\leq m\}$ not being natural and $a\in A$ with $ae_{i}\notin\linspan(\{e^{i}\}).$ Nevertheless, this continues telling us that, given $m$ different minimal ideals generated each one by an idempotent in an evolution algebra, it is always possible to construct an evolution subalgebra $A':=\alg(\{e^{k}\mid 1\leq k\leq m\})=\linspan(\{e^{k}\mid 1\leq k\leq m\})$ of $A$ of dimension $m$ having the set formed by the chosen generators of these $m$ different minimal ideals of the evolution algebra $A$ as its natural basis. One interesting questions to analyze in the future is when $A'=A.$

Let $A$ be an algebra over a field $\mathbb{K}.$ If $e\in A$ is just a minimal idempotent, we cannot assure that $Ae=\mathbb{K}e.$ We call the dimension of the element $a\in A,$ $\dim(a),$ to the dimension of the ideal $\id(\{a\})$ generated by $a.$

\begin{definicion}\label{86}
Let $A$ be an algebra. A basis $B$ is an \textbf{orthogonal bipartite basis} if there exist nonempty subsets $C_{1},C_{2}\subseteq B$ such that we can put $B=C_{1}\cup C_{2}$ with $c_{1}c_{2}=0$ for all $c_{1}\in C_{1}$ and $c_{2}\in C_{2}$ (we represent this writing $C_{1}\bot C_{2}$). We say that $c\in A$ is \textbf{$n$-natural} if there exists an orthogonal bipartite basis $B=C\cup C'$ with $\Card(C)=n$ such that $c\in C$ and $\id(\{c\})=\linspan(\{C\}).$ 
\end{definicion}

As happened with natural elements, we note that this definition implies that $0$ can never be a $n$-natural element.

\begin{definicion}\label{87}
Let $A$ be an algebra and $E$ a subspace of $A.$ We say that $U$ is an \textbf{orthogonal complement of $E$} in $A$ if $E\bot U$ ($eu=0$ for all $e\in E$ and $u\in U$) and $A=E\oplus U.$
\end{definicion}

\begin{teorema}\label{88}
Let $A$ be an algebra and $e\in A$ a minimal idempotent such that $\id(\{e\})$ is an evolution ideal with a natural basis $B$ verifying that there exists $c\in B$ such that $be=0$ for all $b\in B\smallsetminus\{c\}$ but $ce\neq0$ and has an orthogonal complement $E\bot\id(\{e\})$ generated by linear combinations by a basis $C'\bot B,$ then $e$ is a $\dim(e)$-natural element. In particular, if $C'$ is a natural basis of $E$ in the sense that, for all pairs of two different elements $c,c'\in C,$ we have that $cc'=0,$ then $e$ is a natural idempotent. On the other hand, if $0\neq a\in A$ is a $\dim(a)$-natural minimal element generating an evolution ideal $\id(\{a\})$ with $\{af,fa\}\neq\{0\}$ for some $f\in \id(\{a\})$ and $\mathbb{K}$ is a $\dim(a)$-FSEANI, there exist a minimal idempotent $e$ such that $\id(\{e\})=\id(\{a\}).$ 
\end{teorema}

In fact, it is not necessary for $e$ to be idempotent but just \textit{idempotent up to scalars}, that is, $ee=ke$ for one scalar $0\neq k\in\mathbb{K}.$

\begin{proof}
Under our hypothesis $C:=(B\smallsetminus\{c\})\cup\{e\}$ is a basis for $\id(\{e\})$ given the fact that $e=\sum_{d\in B}r_{d}d$ and $0\neq ce=\sum_{d\in B}r_{d}cd$ implies $\proj_{B,c}(e)\neq0$ because $B$ is a natural basis. As $\linspan(C')=E\bot\id(\{e\})=\linspan(C),$ we have that $C\cup C'$ is a basis for $A$ such that $C\bot C'.$ Hence, $C\cup C'$ is an orthogonal bipartite basis with $\Card(C)=\dim(e)$ such that $e\in C$ and $\id(\{e\})=\linspan(C).$

On the other hand, suppose $0\neq a\in A$ is a $\dim(a)$-natural minimal element and $\mathbb{K}$ is a $\dim(a)$-FSEANI. Then $\id(\{a\})$ is a minimal ideal and $\id(\{a\})^{2}\neq\{0\}$ as a consequence of the existence of $f.$ Also, $\{0\}$ is obviously the only proper ideal of $\id(\{a\})$ considered as an algebra given the fact that $A=\id(\{a\})\oplus E$ and $E\bot\id(\{a\}).$ Thus, $\id(\{a\})$ is a simple evolution algebra of dimension $\dim(a)$ over the $\dim(a)$-FSEANI $\mathbb{K}$ so, by Construction \ref{B} and the definitions following it, there exists a minimal idempotent $e$ such that $\id(\{e\})=\id(\{a\})$ and the theorem is proved.
\end{proof}

The existence of the kind of bases described on the previous theorem could be developed through a further study about how are the ideals generated by one element and its bases. For example, some conditions of movement, change or replacement of the elements through the nonassociative operation giving rise to the elements of the basis could be asked.

We will speak about unicity of natural bases in the sense given to this notion in \cite{bcm}. Extending this notion we will be able to count different natural bases up to the obvious isomorphisms of natural bases given by the multiplications of the elements in the basis by non-zero scalars and the permutations (reordering) of the elements in the basis.

\begin{definicion}\cite{bcm}\label{89}
Let $A$ be an evolution algebra over a field $\mathbb{K}.$ We say that $A$ has an \textbf{unique natural basis} if the subgroup of $\Aut_{\mathbb{K}}(A)$ formed by the elements that map natural bases into natural bases is $S_{n}\rtimes(K^{\times})^{n}.$
\end{definicion}

Now will examine the cases in which the natural basis is unique for an evolution algebra.

\begin{teorema}\cite{bcm}\label{C}
Let $A$ be a non-degenerate evolution algebra over $\mathbb{K}.$ Then $A$ has a unique natural basis if and only if there exists a natural basis $B$ such that, for any pair of different vectors $u,v\in B,$ we have that $u^{2}$ and $v^{2}$ are linearly independent.
\end{teorema}

The next definition based on the previous theorem and generalizing \cite[Definition 3.4]{ymv} will be useful.

\begin{definicion}\cite{bcm}
Let $A$ be an evolution algebra. We say that $A$ has the property $2$LI if, for every pair of different vectors $e_{i},e_{j}$ of every natural basis $B,$ the set $\{e_{i}^{2},e_{j}^{2}\}$ is linearly independent.
\end{definicion}

As it is stated in \cite{bcm}, the previous definition is consistent because Theorem \ref{C} tells us that it does not depend on the selected natural basis. We can also opt for the way suggested by the next definition and the result following it. In fact, we can restate the previous theorem.

\begin{proposicion}
Let $A$ be a non-degenerate evolution algebra over $\mathbb{K}.$ Then $A$ has a unique natural basis if and only if $A$ has the property $2$LI.
\end{proposicion}

\begin{definicion}
An algebra $A$ is said to be \textbf{perfect} if $A^{2}=A.$
\end{definicion}

\begin{teorema}\cite[Theorem 4.4]{el}
Every perfect evolution algebra has a unique natural basis.
\end{teorema}

We include here the next definition which was suggested in an unpublished version of the work realized during the preparation of the preprint \cite{bcm} so the credits for its first introduction should be recognized originally for their authors.

\begin{definicion}
Let $A$ be an evolution algebra of dimension $n.$ We say that $A$ has the property $m$LI if $m$ is the greatest number such that there exist a natural basis $B$ of the algebra $A$ having $m$ different vectors $e_{1},\dots e_{m}$ verifying that the set $\{e_{1}^{2},\dots e_{m}^{2}\}$ is linearly independent.
\end{definicion}

The next corollary, which comes again in debt with the work realized in \cite{bcm}, follows immediately from the previous definitions.

\begin{corolario}
Let $A$ be an evolution algebra of dimension $n$ with the property $2$LI. Then, for every natural basis $B$ and for every pair of different vectors $u$ and $v$ of $B,$ we have that $u^{2}$ and $v^{2}$ are linearly independent so, by Theorem \ref{C}, $A$ has a unique natural basis. As a consequence, under the same previous hypothesis, there exist natural idempotents if and only if every natural basis (there is just one natural basis) has an element $e$ such that $e^2=e.$ Also, an evolution algebra of dimension $n$ is perfect if and only if it has the property $n$LI.
\end{corolario}

\begin{definicion}
Let $A$ be an algebra and $B$ a subset of $A.$ We say that $B$ \textbf{ramifies (through the product via a set of multipliers $D\subseteq A$ by the left) towards a set $C$ in one iteration} if there exist a set of elements $D$ contained in $A$ such that, for every $c\in C,$ there exist $d\in D$ and $b\in B$ such that $c=db$ (equivalently, if for the equation $c=xb$ there exists a solution in $D$). We say that $B$ \textbf{ramifies (through the product via a sequence of sets of multipliers $D_{1},\dots,D_{n}\subseteq A$ by the left) towards a set $C$ in $n$ iterations} if there exist a sequence of $n$ sets of elements $D_{1},\dots,D_{n}$ contained in $A$ such that, for every $c\in C,$ there exist a sequence of $n$ elements verifying $d_{i}\in D_{i}$ and $b\in B$ such that $c=d_{n}(\cdots(d_{2}(d_{1}b))).$ Moreover, we say that $B$ \textbf{ramifies inside itself (through the product via a sequence of sets of multipliers $D_{1},\dots,D_{n}\subseteq B$ by the left) towards a set $C$ in $n$ iterations} if there exist a sequence of $n$ sets of elements $D_{1},\dots,D_{n}$ contained in $B$ such that, for every $c\in C,$ there exist a sequence of $n$ elements verifying $d_{i}\in D_{i}$ and $b\in B$ such that $c=d_{n}(\cdots(d_{2}(d_{1}b))).$ Also, we say that $B$ \textbf{ramifies homogeneously (through the product via a set of multipliers $D\subseteq A$ by the left) towards a set $C$ in $n$ iterations} if there exist a set of elements $D$ contained in $A$ such that, for every $c\in C,$ there exist a sequence of $n$ elements verifying $d_{i}\in D$ and $b\in B$ such that $c=d_{n}(\cdots(d_{2}(d_{1}b))).$ In addition, we say that $B$ \textbf{ramifies homogeneously inside itself (through the product via a set of multipliers $D\subseteq B$ by the left) towards a set $C$ in $n$ iterations} if there exist a set of elements $D$ contained in $B$ such that, for every $c\in C,$ there exist a sequence of $n$ elements verifying $d_{i}\in D$ and $b\in B$ such that $c=d_{n}(\cdots(d_{2}(d_{1}b))).$ In general, in any of these cases and if we are not interested in the number of iterations necessary or the sets involved in the process, we just say that $B$ \textbf{ramifies towards} $C.$
\end{definicion}

\begin{proposicion}
Let $A$ be an evolution algebra with a natural basis $B$ rafimifying inside itself towards itself in one iteration. Then all the elements $e\in B$ are natural idempotents.
\end{proposicion}

\begin{proof}
We fix $e\in B$ arbitrary. Then, as $B$ is a natural basis, $e=xb$ with $x,b\in B$ can only have the solution $x=e=b$ and, as $B$ rafimies inside itself towards itself in one iteration, the proposition is proved.
\end{proof}

We have immediately the following two easy corollaries.

\begin{corolario}
Let $A$ be an evolution algebra verifying the property $2$LI (with a unique natural basis $B$ therefore). Then all the elements $e\in B$ are natural idempotents if every (natural) basis $C$ of $A$ ramifies inside itself towards a natural basis in one iteration.
\end{corolario}

\begin{proof}
As $B$ is in particular a (natural) basis, our hypothesis implies that $B$ ramifies towards $B$ (as a consequence of the unicity) in one iteration. We fix $e\in B$ arbitrary. Then $e=xb$ with $x,b\in B$ only have the solution $x=e=b$ and, as $B$ rafimies inside itself towards itself in one iteration, the proposition is proved.
\end{proof}

Using exactly the same proof, we can obtain the next result.

\begin{corolario}
Let $A$ be an evolution algebra verifying that every (natural) basis $C$ of $A$ ramifies inside itself towards a fixed natural basis $B$ in one iteration. Then all the elements $e\in B$ are natural idempotents.
\end{corolario}

\begin{definicion}
An evolution algebra $A$ is \textbf{$n,m$-natural for idempotents} if there exist $m$ natural bases with $n$ non-zero idempotents each one. An evolution algebra is \textbf{$n,m$-conatural for idempotents} if there exist $m$ non-zero idempotents which appear in $n$ natural bases each one. An evolution algebra $A$ is \textbf{$n$-innatural for idempotents} if every non-zero idempotent is in $n,$ and just $n,$ natural bases. An evolution algebra $A$ is \textbf{$n$-surnatural for idempotents} if every natural basis has $n,$ and just $n,$ non-zero idempotents. An evolution algebra $A$ is \textbf{$n,m$-binatural for idempotents} if every natural basis has $n,$ and just $n,$ non-zero idempotents and every non-zero idempotent is in $m,$ and just $m,$ natural bases. 
\end{definicion}

Immediately, we obtain the next proposition.

\begin{proposicion}
Let $A$ be an evolution algebra. Then $A$ is $n,m$-binatural for idempotents if and only if $A$ is $m$-innatural for idempotents and $n$-surnatural for idempotents. If $A$ has unique natural basis, then $A$ is $n,1$-natural if and only if $A$ is $1,n$-natural so, in this case, the unique basis $B$ has $n$ non-zero idempotents and there exist $n$ non-zero idempotents appearing in one natural basis (which can only be $B$) each one.
\end{proposicion}

Also, using the previous corollary we have the following results.

\begin{corolario}\label{103}
Let $A$ be an evolution algebra of dimension $n$ verifying that every (natural) basis $C$ of $A$ ramifies inside itself towards a fixed natural basis $B$ in one iteration. Then $A$ is $1$-innatural and $n$-surnatural so $A$ is $n,1$-binatural.
\end{corolario}

We remember that every natural idempotent of an evolution algebra $A$ is a minimal idempotent generating an ideal of dimension $1$ of $A.$ Now we pass to relate the natural idempotents with the socle. First, we count how many natural idempotents we are able to find in the algebra.

\begin{definicion}\label{104}
Let $A$ be an evolution algebra. The \textbf{socle dimension of natural idempotency of $A$} is the number of different (subspaces of $A$ generated by) natural idempotents in $A$ and we denote it by $\sdni(A).$ On the other hand, the \textbf{socle dimension of non-natural idempotency of $A$} is the number of different minimal idempotents in $A$ which are not natural and we denote it by $\msdnonni(A).$
\end{definicion}

We will need some notation in what follows.

\begin{notacion}
Let $A$ be an evolution algebra. We divide the set $\MinId(A)$ of all the minimal ideals of $A$ in two disjoint subsets so $$\MinId(A)=\ExMinId(A)\cup\InMinId(A),$$ where $\ExMinId(A)$ is formed by all the minimal ideals $E$ of $A$ having a natural basis which can be extended to a natural basis of $A$ (i.e., extension evolution ideals) and $\InMinId(A)$ is the set formed by the minimal ideals of $A$ not having these kind of bases (i.e., the complementary of $\ExMinId(A)$ in $\MinId(A)$). We denote $\natidem(A)$ the set of natural idempotens of $A$ and $\mnonnatidem(A)$ the set of minimal idempotens which are not natural (i.e., the complement of $\natidem(A)$ in the set $\minidem(A)$ of all the minimal idempotens in $A$). Similarly we denote $\NonNatMinId(A)$ the set of all minimal ideals of $A$ which cannot be generated by a natural idempotent and $\NatMinId(A)$ the set of all (necessarily minimal) ideals of $A$ which are generated by a natural idempotent.
\end{notacion}

We note that every ideal $N\in\NatMinId(A)$ has necessarily dimension $1$ and is just the linear space generated by the natural idempotent generating it. Now we sum the subspaces generated by the natural idempotents in an evolution algebra.

\begin{definicion}\label{106}
Let $A$ be an evolution algebra over a field $\mathbb{K}.$ The \textbf{socle of natural idempotency of $A$}, denoted by $\soc\nid(A),$ is defined as the sum of all the $\mathbb{K}e,$ where $e$ is a natural idempotent in $A.$ That is, $$\soc\nid(A):=\sum_{e\in\natidem(A)}\linspan(\{e\}).$$
\end{definicion}

We will prove a proposition about the elements in $\NatMinId(A).$

\begin{proposicion}\label{107}
Let $A$ be an evolution algebra and $N\in\NatMinId(A).$ Then there exist only one idempotent generating $N.$
\end{proposicion}

\begin{proof}
Suppose that there exist two idempotents $e$ and $u$ generating $N.$ Then, as $\dim(N)=1$ for being generated by a natural idempotent, we have $e^{2}=e=ku.$ But $e$ is idempotent so $ku=e=e^{2}=ee=(ku)(ku)=k^{2}u^{2}=k^{2}u,$ which implies $k=k^{2}.$ As $k$ has to be non-zero, this implies $k=1.$ Thus, $e=u,$ as we wanted to prove.
\end{proof}

We have just proved that, for every $N\in\NatMinId(A),$ there exist only one idempotent generating $N.$ We will call this idempotent $e_{N}$ so $N=\id(\{e_{N}\})=\linspan(\{e_{N}\}).$ Then, $$\soc\nid(A)=\sum_{N\in\NatMinId(A)}\linspan(\{e_{N}\})=\sum_{N\in\NatMinId(A)}N.$$ We can also write $$\soc(A)=\left(\sum_{N\in\NatMinId(A)}N\right)+\left(\sum_{I\in\NonNatMinId(A)}I\right)$$ so $$\soc(A)=\soc\nid(A)+\left(\sum_{I\in\NonNatMinId(A)}I\right).$$

We note that $\mathbb{K}e=\linspan(\{e\})=\id(\{e\})=Ae.$ Thus, the socle dimension of natural idempotency of the evolution algebra $A$ is the number of (non-necessarily  linearly independent) summands in the socle of natural idempotency of $A.$ If those summands are all linearly independent, then we say that the socle dimension of natural idempotency of $A$ is \textbf{faithful} for the socle of natural idempotency of $A.$ If not, we define the \textbf{faithful socle dimension of natural idempotency of the socle of natural idempotency of $A$} as the maximal number of linearly independent summands in the definition of the socle of natural idempotency of $A,$ that is, the dimension of $\soc\nid(A)$ as a vector space. When all the natural idempotents of the algebra $A$ are in the same natural basis, the socle dimension of natural idempotency of $A$ is faithful for the socle of natural idempotency of $A$ because then every linear subspace $\linspan(\{e\})$ in the sum has to be linearly independent from the rest of summands.

Now we can prove an important theorem about the socle.

\begin{teorema}\label{108}
Let $A$ be an evolution-minimal-ideally simple algebra of dimension $n$ over a field $\mathbb{K}$ and verifying the property $2$LI. Define $\sdni(A)=s.$ Then $\dim(\sum_{N\in\NatMinId(A)}N)=s$ and, therefore, $\dim(\sum_{I\in\NonNatMinId(A)}I)\leq n-s.$ If, in addition, we suppose that $\mathbb{K}$ is an $m$-FSEANI for all $m\leq n-s,$ then, for every $I\in\NonNatMinId(A)$ there exist an strictly non-natural idempotent $a_{I}$ such that $I=\id(\{a_{I}\})$ and we can write $$\soc(A)=\left(\sum_{N\in\NatMinId(A)}\linspan(\{e_{i}\})\right)+\left(\sum_{I\in\NonNatMinId(A)}\id(\{a_{i}\})\right).$$
\end{teorema}

\begin{proof}
As the algebra $A$ verifies the property $2$LI, the algebra $A$ has a unique (up to permutations and products by scalars of $\mathbb{K}$) natural basis $B.$ Then all the $s$ natural idempotents are in $B.$ Take the set of natural idempotents $S\subseteq B.$ Hence, $\sum_{N\in\NatMinId(A)}N=\sum_{N\in\NatMinId(A)}\linspan(\{e_{i}\})=\linspan(S)$ so $\dim(\sum_{N\in\NatMinId(A)}N)=s.$ As $\dim(A)=n,$ this implies $\dim(\sum_{I\in\NonNatMinId(A)}I)\leq n-s.$ Now, if, in addition, we have that $\mathbb{K}$ is an $m$-FSEANI for all $m\leq n-s,$ then, as every $I\in\NonNatMinId(A)$ is simple ($A$ is an evolution-minimal-ideally simple algebra) and verifies $\dim(I)\leq n-s,$ we obtain that, for every $I\in\NonNatMinId(A),$ there exist an idempotent $a_{I}$ such that $I=\id(\{a_{I}\})$ and we can write $$\soc(A)=\left(\sum_{N\in\NatMinId(A)}\linspan(\{e_{i}\})\right)+\left(\sum_{I\in\NonNatMinId(A)}\id(\{a_{i}\})\right).$$\end{proof}

We hope that the notions introduced along this work can serve in the future as useful tools to extend and pose new results having a flavours similar to the previous theorem in the theory and study of evolution algebras and beyond. It seems particularly important to find ways and constraints which could ensure that the behaviour of the part of the socle expressed by the rightmost sum of the expression above is under our control and admits nicer representations in order to get closer to a framework compatible with an adequate classification of evolution algebras.

\printbibliography[title=References]\addcontentsline{toc}{chapter}{References}

\end{document}